\documentclass {article}
\usepackage{authblk}
\usepackage{etex}
\usepackage[utf8]{inputenc}
\usepackage[english]{babel}
\usepackage{amsmath}
\usepackage{amsthm}
\usepackage{amssymb}
\usepackage{sectsty}
\usepackage{titlesec}
\usepackage{color}
\usepackage[color,matrix,arrow]{xy}
\usepackage{amsgen}
\usepackage{amstext}
\usepackage{amsbsy}
\usepackage{amsopn}
\usepackage{amsfonts}
\usepackage{eepic}
\usepackage{graphicx}
\usepackage{epsf}
\usepackage{pstricks}
\xyoption{all}

\usepackage{pgf}
\usepackage{tikz}

\usepackage{hyperref}

\newcommand{\footrecall}[1]{%
}

\def\mapright#1{\smash{\mathop{\longrightarrow}\limits^{#1}}}
\def\tra#1{\smash{\mathop{\mid\kern
-1pt\joinrel\relbar\joinrel\relbar}\limits^{*}_{#1}}}
\def\longtra#1{\smash{\mathop{\mid\kern
-1pt\joinrel\relbar\joinrel\relbar\joinrel\relbar}\limits^{*}_{#1}}}
\def\vlongtra#1{\smash{\mathop{\mid\kern
-1pt\joinrel\relbar\joinrel\relbar\joinrel\relbar\joinrel\relbar}\limits^{*}_{#1}}}
\def\vvlongtra#1{\smash{\mathop{\mid\kern
-1pt\joinrel\relbar\joinrel\relbar\joinrel\relbar\joinrel\relbar\joinrel\relbar}\limits^{*}_{#1}}}
\def\vvvlongtra#1{\smash{\mathop{\mid\kern
-1pt\joinrel\relbar\joinrel\relbar\joinrel\relbar\joinrel\relbar\joinrel\relbar\joinrel\relbar}\limits^{*}_{#1}}}
\def\etra#1{\smash{\mathop{\mid\kern
-1pt\joinrel\relbar\joinrel\relbar}\limits_{#1}}}
\def\mapleft#1{\smash{\mathop{\longleftarrow}\limits^{#1}}}

\def\vlongrightarrow{\relbar\joinrel\longrightarrow}

\def\longmapright#1{\smash{\mathop{\vlongrightarrow}\limits^{#1}}}

\newcommand{\Mon}{\mbox{\rm Mon}}
\def\A{{\cal{A}}}

\def\iff{\Leftrightarrow}
\def\Rw{\Rightarrow}
\def\oo{\overline}

\def\wt{\widetilde}
\def\wh{\widehat}

\def\B{{\cal{B}}}
\def\C{{\cal{C}}}

\def\im{\mbox{Im}\,}

\def\cl{\mbox{Cl}\,}

\def\cay{\mbox{Cay}}

\def\ker{\mbox{Ker}\,}

\def\mt{\mbox{MT}}

\def\fold{\mbox{Fold}}

\def\P{{\cal{P}}}

\def\R{\mathrel{{\mathcal R}}}

\def\p{\varphi}

\def\inv{^{-1}}
\def\la{\langle}
\def\ra{\rangle}

\def\J{{\cal{J}}}
\def\bi{\begin{itemize}}
\def\ei{\end{itemize}}
\def\beq{\begin{equation}}
\def\eeq{\end{equation}}

\titleformat*{\section}{\large\bfseries}
\titleformat*{\subsection}{\normalsize \bfseries}

\title{Eraser morphisms and membership problem in groups and monoids}
\author[1]{Daniele D'Angeli \thanks{daniele.dangeli@unicusano.it}}
\author[2]{Emanuele Rodaro \thanks{emanuele.rodaro@gmail.com}}
\author[3]{Pedro V. Silva \thanks{pvsilva@fc.up.pt}}
\author[4]{Alexander Zakharov \thanks{zakhar.sasha@gmail.com}}
\affil[1]{Universita degli Studi Niccolo Cusano - Via Don Carlo Gnocchi, 3, 00166 
Roma, Italy}
\affil[2]{Department of Mathematics, Politecnico di Milano, Piazza Leonardo da Vinci, 32, 20133 Milano, Italy}
\affil[3]{Centre of Mathematics, University of Porto, R. Campo Alegre 687, 4169-007 Porto,
Portugal}
\affil[4]{Chebyshev Laboratory, St Petersburg State University, 14th Line 29B, Vasilyevsky Island, St.Petersburg, 199178, Russia, and The Russian Foreign Trade Academy, 4a Pudovkina street, 119285, Moscow, Russia.}

\date{}

\begin{document}

\newtheorem{theorem}{Theorem}[section]
\newtheorem{lemma}[theorem]{Lemma}
\newtheorem{question}[theorem]{Question}
\newtheorem{proposition}[theorem]{Proposition}
\newtheorem{corollary}[theorem]{Corollary}

\theoremstyle{definition}
\newtheorem{definition}[theorem]{Definition}
\newtheorem{example}[theorem]{Example}

\newcommand{\ophi}{\overline{\varphi}}
\newcommand{\opsi}{\overline{\psi}}

\maketitle

%
%
%
%

%
%
%
%

\begin{abstract}
We develop the theory of fragile words by introducing the concept of eraser morphism and extending the concept to more general contexts such as (free) inverse monoids. We characterize the image of the eraser morphism in the free group case,  and show that it has decidable membership problem. We establish several algorithmic properties of the class of finite-$\J$-above (inverse) monoids. We prove that the image of the eraser morphism in the free inverse monoid case (and more generally, in the finite-$\J$-above case) has decidable membership problem, and relate its kernel to the free group fragile words. 
\end{abstract}

\section{Introduction}

In \cite{puzzles}, Demaine {\em et al.} introduced the concept of a {\it fragile word} in the context of free groups: a word $w \in FG_A$ is called fragile if it is in the kernel of the canonical projection $\epsilon_a:FG_A \to FG_{A\setminus \{ a\}}$ for every  $a \in A$. One can systematically build fragile words with the help of nested commutators, which provide a quadratic upper bound for the length of the shortest nontrivial fragile words on $n$ letters.

There is a geometric interpretation of fragile words in terms of {\it Brunnian links}, i.e. collections of $n$ loops that are linked but such that the removal of any loop unlinks the rest, see \cite{puzzles} and \cite{GG}. An example is the famous Borommean link (Borommean rings), which is a Brunnian link with 3 loops.
It was proved in \cite{GG} that the quadratic bound is actually the best possible, i.e. the shortest non-trivial fragile word in a free group of rank $n$ has quadratic length in $n$, given by a precise formula, and in particular appropriate nested commutators are the shortest non-trivial fragile words, for each $n$. The proof in \cite{GG} uses the terminology of Brunnian links, with fragile words of shortest length corresponding to minimal Brunnian links, but, in fact, the proof is purely combinatorial.  

In \cite{puzzles}  Demaine {\em et al.} also established connections with the nice combinatorial problem of {\em picture-hanging puzzles}. In the picture-hanging puzzle we are to hang a picture on a wall using a string and $n$ nails, so that the string loops around $n$ nails and the removal of any nail results in a fall of the picture. Since the fundamental group of a plane with $n$ points (nails) removed is isomorphic to the free group of rank $n$, there is a one-to-one correspondence between words in a free group of rank $n$ and different ways to hang a picture on $n$ nails, considered up to homotopy, with fragile words corresponding precisely to the hangings such that the removal of any nail results in a fall of the picture. 
We refer to \cite{puzzles} for more details about such interpretation, and information about more general picture-hanging puzzles.

On the other hand, the first two authors established in \cite{fragile} a new motivation. There exists a connection with the theory of automaton groups, through the freeness problem: given an invertible transducer with the extended alphabet, the shortest nontrivial relator (if the automaton group is not free) is always given by a fragile word.

The present paper aims at developing the theory of fragile words in two directions:
\begin{itemize}
\item
Introducing and studying the concept of eraser morphism.
\item
Extending the concept of fragile word to more general contexts such as (free) inverse monoids.
\end{itemize} 

The eraser morphism $\epsilon:FG_A \to \prod_{a \in A} FG_{A\setminus \{ a\}}$ is defined as $\epsilon = \prod_{a \in A} \epsilon_a$ and its kernel is constituted by the fragile words of $FG_A$. We can generalize this notion to a group defined by a fixed finite presentation. We note that such erasing morphisms have been considered early in combinatorial group theory. A celebrated example is Magnus' Freiheitssatz, stating that in the case of a one-relator group, the image of each $\epsilon_a$ is free on $A \setminus \{ a\}$ \cite{frei}. It is only natural that, since the kernel of the eraser morphism $\epsilon$ consists of the fragile words, one considers also its image. One of the main results of this paper characterizes the image of $\epsilon$ in the free group case,  and shows that it has decidable membership problem.

Inverse monoids generalize groups in the following sense: as groups are, up to isomorphism, sets of permutations on a fixed set, closed under composition and inverse functions, inverse monoids are, up to isomorphism, sets of partial injective transformations on a fixed set, closed under compositions and inverse functions, and containing the identity. This makes inverse monoids virtually ubiquitous in areas such as geometry or topology. We also discuss fragile words and eraser morphisms in the context of inverse monoids. In order to do so, we study the (good) algorithmic properties of the class of finite-$\J$-above (inverse) monoids (which include a solvable word problem), a subject of independent interest. 

Among other results, we prove that the image of $\epsilon$ in the free inverse monoid case (and more generally, in the finite-$\J$-above case) has decidable membership problem, and relate its kernel to the free group fragile words. We also consider finiteness conditions such as having finite $\R$-classes, and relate the property in the original inverse monoid with the corresponding quotients featuring the eraser morphism.

The paper is organized as follows: in Section 2 we present background concepts and results, including the main tools of geometric inverse semigroup theory and some rudiments of automata and language theory. Section 3 is devoted to the algorithmic properties of finite-$\J$-above (inverse) monoids. Finally, in Section 4 we study fragile words and the eraser morphism for both free groups and subclasses of inverse monoids.

\section{Preliminaries}

\subsection{Automata}\label{sec: automata}

Let $A$ be a finite set (usually called an {\em alphabet}). An $A$-{\em automaton} is a quadruple of the form $\A = (Q,I,T,E)$, where:
\bi
\item
$Q$ is a set (states),
\item
$I,T \subseteq Q$ (initial and terminal states, respectively),
\item
$E \subseteq Q \times A \times Q$ (transitions).
\ei
The automaton $\A$ is finite if $Q$ is finite. A {\em path} in $\A$ is a sequence of the form
\beq
\label{path}
q_0 \mapright{a_1} q_1 \mapright{a_2} \ldots \mapright{a_n} q_n
\eeq
with $n \geq 1$ and $(q_{j-1},a_j,q_j) \in E \cup \{ (q,1,q) \mid q \in Q\}$ for $j = 1,\ldots,n$. Its {\em label} is the word $a_1a_2\ldots a_n \in A^*$. The path (\ref{path}) is {\em successful} if $q_0 \in I$ and $q_n \in T$. The {\em language} of $\A$, denoted by $L(\A)$, is the set of labels of all successful paths in $\A$. The automaton $\A$ is {\em trim} if every vertex occurs in some successful path.

A subset of $A^*$ is an $A$-{\em language}. An $A$-language $L$ is {\em rational} if $L = L(\A)$ for some finite $A$-automaton $\A$. We may always assume that a rational language is recognized by a finite deterministic automaton. An $A$-automaton $\A = (Q,I,T,E)$ is {\em deterministic} if $|I| = 1$ and 
$$(p,a,q), (p,a,q') \in E \, \Rw \, q = q'.$$
We summarize in the following result the properties of rational languages which are relevant for us:

\begin{proposition}
\label{rat}
\bi
\item[(i)] The set of rational $A$-languages is closed under the boolean operations and all the constructions involved are effective.
\item[(ii)] Given rational $A$-languages $L,L'$, it is decidable whether or not: $L = \emptyset$, $L = L'$.
\ei
\end{proposition}

A {\it morphism} $\varphi$ from an $A$-automaton $\A_1=(Q_1,I_1,T_1,E_1)$ to an $A$-automaton $\A_2=(Q_2,I_2,T_2,E_2)$ is a pair of maps
$\varphi: Q_1 \rightarrow Q_2$ and $\varphi: E_1 \rightarrow E_2$,
which satisfy the three properties: $\varphi(I_1) \subseteq I_2$, $\varphi(T_1) \subseteq T_2$, and $\varphi(p,a,q)=(\varphi(p),a,\varphi(q))$ for every $(p,a,q) \in E_1$.

Two $A$-automata $\A = (Q,I,T,E)$ and $\A' = (Q',I',T',E')$ are {\em isomorphic} if there exists a bijection $\p:Q \to Q'$ such that $I\p = I'$, $T\p = T'$ and 
$$(p,a,q) \in E \iff (p\p,a,q\p) \in E'$$
holds for all $p,q \in Q$ and $a \in A$. 

Given an alphabet $A$, we define a set $A^{-1} = \{ a^{-1} \mid a \in A\}$ of formal inverses of $A$ and write $\wt{A} = A\cup A\inv$. An $\wt{A}$-automaton $\A = (Q,I,T,E)$ is {\em involutive} if
$$(p,a,q) \in E \iff (q,a\inv,p) \in E$$
holds for all $p,q \in Q$ and $a \in A$. If $|T| = 1$ and $\A$ is also deterministic and trim, 
we call it an {\em inverse} automaton. Inverse automata are known to be minimal, which implies they are fully determined (up to isomorphism) by their language.

Let $\A = (Q,I,T,E)$ be a finite involutive $\wt{A}$-automaton. We define a finite involutive $\wt{A}$-automaton $\fold(\A)$ by successively identifying distinct vertices $q,q'$ whenever there exist edges
$$q \mapleft{a} p \mapright{a} q'$$
with $a \in \wt{A}$. This procedure establishes a confluent algorithm which terminates after finitely many steps, producing a deterministic automaton (inverse if $|I| = |T| = 1$ and the underlying graph is connected).

\subsection{Rational and recognizable subsets}
	
	Suppose $M$ is a finitely generated monoid. In analogy to rational languages in a free monoid, one can define the set of {\it rational subsets} in $M$ to consist of all the subsets of $M$ which can be obtained from the finite ones by taking unions of two subsets, products of two subsets and Kleene star of a subset (i.e., passing to the submonoid generated by the given subset). In particular, rational subsets of free monoids are precisely the rational languages, and a subgroup of a group is a rational subset if and only if it is finitely generated (due to Anisimov and Seifert \cite[Theorem III.2.7]{Ber}). Let $\alpha: A^* \rightarrow M$ be a surjective homomorphism, for some finite alphabet $A$. It turns out that a subset $K$ of $M$ is rational if and only if there exists a rational language $L$ in $A^*$ such that $K=\alpha(L)$, see \cite[Proposition 1.7, p.223]{Sak}; this property is sometimes taken as the definition of a rational subset.
	\par 
	 We say that $X \subseteq M$ is a {\em recognizable subset} of $M$ if there exists a homomorphism $\theta:M \to K$ to some finite monoid $K$ satisfying $X = \theta^{-1}\theta(X)$.  Given a surjective homomorphism  $\alpha: A^* \rightarrow M$, it turns out that a subset $X$ of $M$ is recognizable if and only if $\alpha^{-1}(X)$ is a rational language \cite[Theorem 2.2, p.247]{Sak}. 
	 A subgroup of a group is a recognizable subset if and only if it has finite index \cite[Proposition 6.1, p.302]{Sak}. \par 
	 Equivalently, a subset $X$ of $M$ is recognizable if and only if the {\em syntactic congruence} $\sim_X$ has finite index (i.e., there are finitely many equivalence classes); this is the congruence on $M$ defined by $u \sim_X v$ if
$$\forall p,q \in M\, (puq \in X \Leftrightarrow pvq \in X);$$ 
see \cite[Theorem 2.3, p.247]{Sak}.
 	Let ${\rm Rat}(M)$ (respectively ${\rm Rec}(M)$) denote the set of all rational (respectively recognizable) subsets of $M$. For a finitely generated monoid $M$ we always have ${\rm Rec}(M) \subseteq {\rm Rat}(M)$ \cite[Proposition III.2.4]{Ber}. 


\subsection{Inverse monoids and Sch\"utzenberger automata}\label{im}

A monoid $M$ is said to be {\em inverse} if it satisfies
$$\forall a \in M\, \exists ! b \in M\; (aba = a \, \wedge\, bab = b).$$
The element $b$ is called the inverse of $a$ and is denoted by $a^{-1}$. The identity of $M$ is denoted by $1$ as usual.

The following alternative characterizations of inverse monoids are well known:

\begin{proposition}
\label{acim}
The following conditions are equivalent for a monoid $M$:
\begin{itemize}
\item[(i)] $M$ is inverse;
\item[(ii)] the idempotents of $M$ commute and $M$ satisfies
$$\forall a \in M\,  \exists b \in M\; aba = a.$$
\item[(iii)] $M$ is isomorphic (as a monoid) to some monoid of partial injective functions containing the inverse functions of its elements.
\end{itemize}
\end{proposition}

One of the features of inverse monoids is the existence of the so-called {\em natural partial order}, a partial order compatible with both product and inversion. If $M$ is an inverse monoid and $E(M) = \{ e \in M \mid e^2 = e\}$, the natural partial order on $M$ is defined by
$$u \leq v\mbox{ if $u = ev$ for some }e \in E(M).$$ 
Equivalently, $u \leq v$ if and only if $u = uu\inv v$.

The class of inverse monoids constitutes a variety ${\cal{I}}$ for the signature $(\cdot, \, ^{-1},1)$ (with arities 2,1,0 respectively). We describe next the free objects of ${\cal{I}}$.

Given a set $A$, recall the notation $\wt{A} = A\cup A\inv$, where $A\inv$ denotes a set of formal inverses of $A$.
We extend the mapping defined $a \mapsto a\inv$ $(a \in A)$ to an involution $\inv$ of $\wt{A}$, which is subsequently extended to an involution of $\wt{A}$ through
$$1\inv = 1,\quad (ua)\inv = a\inv u\inv\; (u \in \wt{A}^*, a\in \wt{A}).$$ 
The {\em free inverse monoid} on a set $A$ can be described as the quotient $FIM_A = \wt{A}^*/\rho_A$, where $\rho_A$ (the {\em Vagner congruence}) denotes the congruence on $\wt{A}^*$ generated by the relation
$$\{ (ww\inv w,w) \mid w \in \wt{A}^*\} \cup \{ (uu\inv vv\inv, vv\inv uu\inv ) \mid u,v \in \wt{A}^*\}.$$
The word problem for $FIM_A$ was solved (independently) in the 1970's by Munn and Scheiblich. We describe next Munn's solution: given $w \in \wt{A}^*$, say $w = a_1\ldots a_n$ $(a_i \in \wt{A})$, we define ${\rm Lin}(w)$ (the linear automaton of $w$) by taking:
\bi
\item
all the prefixes of $w$ as states;
\item
1 and $w$ as initial and terminal states, respectively;
\item
transitions $a_1\ldots a_{j-1} \mapright{a_j} a_1\ldots a_j$ and $a_1\ldots a_{j-1} \mapleft{a_j\inv} a_1\ldots a_j$ for $j = 1,\ldots,n$.
\ei
Thus ${\rm Lin}(w)$ looks like
$$\to 1 \mapright{a_1} a_1 \mapright{a_2} a_1a_2 \mapright{a_3} \ldots \mapright{a_n} a_1\ldots a_n = w \to$$
where we omit the opposite edges (${\rm Lin}(w)$ is involutive).
The {\em Munn tree} of $w$ is defined as
$$\mt(w) = \fold({\rm Lin}(w)).$$
An alternative construction uses the concept of Cayley graph of the free group $FG_A$. If $M$ is a monoid generated by a subset $A$, the {\em Cayley graph} $\cay_A(M)$ has the elements of $M$ as vertices and edges of the form $m \mapright{a} ma$ for all $m \in M$ and $a \in A$. Then the underlying graph of $\mt(w)$ can be described as the subgraph of $\cay_{\wt{A}}(FG_A)$ spanned by the path $1 \mapright{w} w \in G$, considering also the opposite edges. Then we take 1 and $w \in G$ as initial and terminal vertices, respectively.

The word problem for $FIM_A$ is now solved through the equivalence
$$u \rho_A v \; \iff\; \mt(u) \cong \mt(v),$$
which holds for all $u,v \in \wt{A}^*$.

Every inverse monoid  is isomorphic to some quotient of some free inverse monoid. Technically, a presentation of inverse monoids is a formal expression of the form ${\rm Inv}\langle A \mid R\rangle$, where $A$ is a set and $R \subseteq \wt{A}^* \times \wt{A}^*$. The inverse monoid defined by this presentation is $\wt{A}^*/(\rho_A \cup R)^{\sharp}$ (we denote by $R^{\sharp}$ the congruence on a monoid $M$ generated by a relation $R \subseteq M \times M$).

The standard technique to studying inverse monoid presentations involves the strongly connected components of the Cayley graph, known as Sch\"utzenberger graphs. The most common definition involves the Green $\cal{R}$-relation. Given an inverse monoid $M$ and $u,v \in M$, we write $u {\cal{R}} v$ if $uu\inv = vv\inv$. Equivalently, $u{\cal{R}}v$ if and only if there exist $x,y \in M$ such that $u = vx$ and $v = uy$. This is an equivalence relation on $M$. 

Assume now that $M$ is defined by the presentation Inv$\langle A \mid R\rangle$, and write $\tau = (\rho_A \cup R)^{\sharp}$. It is known that:
\bi
\item
the edge $u\tau \mapright{a} (ua)\tau$ of $\cay_{\wt{A}}(M)$ admits an opposite edge if and only if $u\tau \,{\cal{R}} \, (ua)\tau$;
\item
the strongly connected components of $\cay_{\wt{A}}(M)$ are the subgraphs induced by the $\cal{R}$-classes of $M$.
\ei
Given $w \in \wt{A}^*$, the {\em Sch\"utzenberger automaton} of $w$ (with respect to the above presentation) is obtained from the Sch\"utzenberger graph containing $u\tau$ by setting $(uu\inv)\tau$ (respectively $u\tau$) as the unique initial (respectively terminal) state. Note that $x \, {\cal{R}} xx\inv$ holds in every inverse monoid. We denote the Sch\"utzenberger automaton of $w$ by $\A(w)$, omitting therefore the presentation when it is clear from the context. Note that Sch\"utzenberger automata are always inverse automata. The following results illustrate the role played by Sch\"utzenberger automata on solving the word problem of the inverse monoid $M$, see \cite{Steph90} and \cite{Steph98}:

\begin{proposition}\label{prop: isomorphic}
\label{order}
Let {\rm Inv}$\langle A \mid R\rangle$ be a presentation and write $\tau =  (\rho_A \cup R)^{\sharp}$. For every $u \in \wt{A}^*$, $L(\A(u)) = \{ v \in \wt{A}^* \mid v\tau \geq u\tau \}$.
\end{proposition}

\begin{proposition}
\label{schutz}
Let {\rm Inv}$\langle A \mid R\rangle$ be a presentation and write $\tau =  (\rho_A \cup R)^{\sharp}$. 
The following conditions are equivalent for all $u,v \in \wt{A}^*$:
\bi
\item[(i)] $u \: \tau \: v$;
\item[(ii)] $\A(u) \cong \A(v)$;
\item[(iii)] $L(\A(u)) = L(\A(v))$;
\item[(iv)] $u \in L(\A(v))$ and $v \in L(\A(u))$.
\ei
\end{proposition}
In \cite{Steph90} it is provided a (confluent) iterative procedure for constructing the Sch\"utzenberger automaton relative to a given finite presentation ${\rm Inv}\langle A \mid R\rangle$ of a word $w\in \wt{A}^*$ via two operations. One that we have already seen in Section~\ref{sec: automata}, that is the folding operation on an involutive $\wt{A}$-automaton. The other operation is called \emph{the elementary expansion}. The elementary expansion applied to an involutive $\wt{A}$-automaton $\B$ consists in adding a path $q_1\mapright{u}q_2$ to $\B$ wherever $q_1\mapright{v}q_2$ is a path in $\B$ and $(u,v)\in R$ or $(v,u) \in R$. Following \cite{Steph98} we may define for any involutive $\wt{A}$-automaton $\B$ a directed system consisting of all objects that can be obtained by an arbitrary sequence of elementary expansions and folding operations. This directed system is a downwardly directed commutative diagram, hence it has a colimit (see \cite{Steph98}). Accordingly to \cite{Steph98} we call this colimit the \emph{closure} of $\B$ with respect to $R$ and it is denoted by $\cl_{R}(\B)$. In the same paper it is shown that $\cl_R({\rm Lin}(w))$ is the Sch\"utzenberger automaton $\A(w)$ of the word $w\in\wt{A}^*$ with respect to ${\rm Inv}\langle A \mid R\rangle$. In particular, since for the free inverse monoid there are no elementary expansions to perform, we get $\cl_R({\rm Lin}(w))=\fold({\rm Lin}(w))$, hence we immediately deduce the result of Munn: $\mt(w) = \fold({\rm Lin}(w))$. 

The following result shows that in case a Sch\"utzenberger automaton is finite, the confluent sequence devised by Stephen is finite and so this automaton is effectively constructible. This result is known, but we include a proof for lack of an adequate reference.

\begin{proposition}
\label{cons}
Let {\rm Inv}$\langle A \mid R\rangle$ be a finite presentation and let $u \in \wt{A}^*$ be such that $\A(u)$ is finite. Then
$\A(u)$ is effectively constructible.
\end{proposition}

\begin{proof}

The Stephen sequence \cite{Steph90} is a sequence $(\A_n(u))_n$ of finite inverse $\wt{A}$-automata with the following properties:
\bi
\item[(P1)]
$\A_1(u) = \mt(u)$;
\item[(P2)]
$\A_n(u)$ is effectively constructible from $R$ and $\A_{n-1}(u)$ for each $n \geq 2$. Namely, $\A_{n}(u)$ is obtained from $\A_{n-1}(u)$ by applying simultaneously all possible instances of elementary expansions and foldings.
\item[(P3)]
$L(\A_1(u)) \subseteq L(\A_2(u)) \subseteq L(\A_3(u)) \subseteq \ldots$
\item[(P4)]
$\displaystyle
L(\A(u)) = \bigcup_{n \geq 1} L(\A_n(u))$.
\ei
We show that, in our case, there exists some $m \geq 1$ such that $L(\A(u)) = L(\A_m(u))$. Indeed, assume that $\A(u) = (Q,q_0,t,E)$ (with $Q$ finite). For every $q \in Q$, fix a path $q_0 \mapright{v_q} q$ for some $v_q \in \wt{A}^*$ (existence follows from $\A(u)$ being trim). We may assume that $v_{q_0} = 1$. Now $v_t \in L(\A(u))$, and for every $(p,a,q) \in E$ we have $v_pav_q\inv v_t \in L(\A(u))$. Since $E$ is finite and in view of (P4), there exists some $m \geq 1$ such that
\beq
\label{cons1}
\{ v_t \} \cup \{ v_pav_q\inv v_t \mid (p,a,q) \in E \} \subseteq L(\A_m(u)).
\eeq
Also by (P4), we have $L(\A_m(u)) \subseteq L(\A(u))$. Conversely, let $w \in L(\A(u))$. Then there exists a path
$$q_0 \mapright{a_1} q_1 \mapright{a_2}  \ldots \mapright{a_n} q_n = t$$
in $\A(u)$, with $a_1,\ldots,a_n \in \wt{A}$ and $w = a_1\ldots a_n$. Let $q'_0$ and $t'$ denote the initial and terminal vertices of $\A_m(u)$, respectively. Write $z_i = v_{q_{i-1}}a_iv_{q_i}\inv v_t$ for $i = 1,\ldots,n$.
Since $\A(u)$ is inverse, it follows easily from (\ref{cons1}) that there exists in $\A_m(u)$ a path of the form
$$q'_0 \mapright{z_1} t' \mapright{v_t\inv} q_0 \mapright{z_2} t' \mapright{v_t\inv} \ldots \mapright{z_{n-1}} t' \mapright{v_t\inv} q'_0 \mapright{z_n} t'.$$
Thus
\beq
\label{cons2}
\begin{array}{l}
v_{q_{0}}a_1v_{q_1}\inv v_tv_t\inv v_{q_{1}}a_2v_{q_2}\inv v_tv_t\inv \ldots v_{q_{n-2}}a_{n-1}v_{q_{n-1}}\inv v_tv_t\inv v_{q_{n-1}}a_nv_{q_n}\inv v_t\\
\hspace{5cm}
= z_1v_t\inv z_2v_t\inv \ldots z_{n-1}v_t\inv z_n \in L(\A_m(u)).
\end{array}
\eeq
Since $\A_m(u)$ is deterministic and involutive, we can successively omit factors of the form $xx\inv$ from any word in $L(\A_m(u))$. Since $v_{q_0} = 1$ and $q_n = t$, it follows from (\ref{cons2}) that $w = a_1\ldots a_n \in L(\A_m(u))$. Therefore $L(\A(u)) \subseteq L(\A_m(u))$ and so $L(\A(u)) = L(\A_m(u))$. 

It follows from (P3) and (P4) that $L(\A_{m+1}(u)) = L(\A_m(u))$. 
Thus, when building the sequence $(\A_n(u))_n$, we must necessarily reach some $k \geq 1$ such that $L(\A_{k+1}(u)) = L(\A_k(u))$, and this equality is decidable by Proposition \ref{rat}.
Now it suffices to show that $L(\A_{k+1}(u)) = L(\A_k(u))$ implies $\A_k(u) \cong \A(u)$.

Indeed, by (P2) we must have $L(\A_{k}(u)) = L(\A_{k+1}(u)) = L(\A_{k+2}(u)) = \ldots$, yielding $L(\A_{k}(u)) = L(\A(u))$ in view of (P3) and (P4). Since inverse automata are minimal, this implies $\A_k(u) \cong \A(u)$ and we are done.
\end{proof}

\section{Membership problem in finite-$\J$-above monoids}

Given a finitely generated monoid $M$ and a subset $L \subseteq M$, the {\it membership problem} (also known as generalized word problem) for $L$ in $M$ asks whether, given an element $g \in M$, one can decide whether $g \in L$. It is known that there exist finitely generated subgroups in the direct product of two free groups with unsolvable membership problem, due to Mihailova \cite{Mih}. Henceforth, we consider the membership problem in the setting of monoids having the finite-$\J$-above property. 
\subsection{Finite-$\J$-above monoids and factors}

Given a monoid $M$ and $u,v \in M$, we say that $u$ is a {\em factor} of $v$ if $v = xuy$ for some $x,y \in M$. Then $M$ is {\em finite-$\J$-above} if every element of $M$ has only finitely many factors. This is equivalent to say that every element of $M$ has only finitely many elements above it in the $\cal{J}$-order (we write $u \geq_{\cal{J}} v$ if $u$ is a factor of $v$). Note that the property of being finite-$\J$-above does not depend at all on any choice of presentation for $M$.

Note that free inverse monoids are always finite-$\J$-above: indeed, if $u$ is a factor of $v$ in $FIM(A)$, then the underlying graph of $\mt(u)$ embeds in the underlying graph of $\mt(v)$. Since both graphs are trees and $\mt(v)$ is finite, this leads us to finitely many choices for $u$ as an element of $FIM_A$.

Note also that direct products of finitely many finite-$\J$-above monoids are finite-$\J$-above. In particular, direct products of finitely many free inverse monoids are finite-$\J$-above.

A monoid $M$ is called {\it finitely recognizable} if singleton subsets of $M$ are recognizable. Equivalently, a monoid $M$ is finitely recognizable if finite subsets of $M$ are recognizable. See \cite{LV} for more details.

A monoid $M$ is called {\it residually finite} if for any distinct elements $s$ and $t$ in $M$ there is a congruence of finite index on $M$ such that $s$ and $t$ are not equivalent under this congruence. In other words, a monoid is residually finite if any pair of distinct elements can be distinguished in a finite quotient of $M$.

\begin{proposition}\label{ff}
	Let $M$ be a monoid and $u \in M$ have finitely many factors. Then the singleton $\{ u \}$ is recognizable. 
\end{proposition}

\begin{proof}

	This was proved in \cite{LV}. We present a short proof here. Indeed, for any $v, w$ which are not factors of $u$ we have that $v$ and $w$ belong to the same syntactic congruence class with respect to $\{u \}$. Since $u$ has finitely many factors, this means that the syntactic monoid of $M$ with respect to $\{ u \}$ is finite, and so $\{ u \}$ is recognizable.
\end{proof}	

\begin{corollary}\label{fr}
	Let $M$ be a finite-$\J$-above monoid. Then $M$ is finitely recognizable.
\end{corollary}

\begin{proposition}\label{wp}
	Let $M$ be a finitely presented finite-$\J$-above monoid. Then $M$ has solvable word problem.
\end{proposition}

\begin{proof}

This was also proved in \cite{LV}. Indeed, by the proof of Proposition \ref{ff} $M$ is residually finite, and residually finite finitely presented monoids have solvable word problem, by \cite[Theorem 6.2]{LV}.

\end{proof}

\begin{proposition}\label{f}
	Let $M$ be a finitely presented finite-$\J$-above monoid. Then it is decidable whether one element of $M$ is a factor of another.
\end{proposition}

\begin{proof}
	Let  {\rm Mon}$\langle A \mid R\rangle$ be the finite presentation for $M$ and $\varphi: A^* \rightarrow M$ be the corresponding epimorphism. Let $u, v \in M$ be given by $w, z \in A^*: \varphi(w)=u, \varphi(z)=v$, and we need to decide whether $u$ is a factor of $v$. \par
	We run two algorithms simultaneously, the first one terminating if and only if $v \in \text{Fact} (u)$, and the second one if and only if $v \notin \text{Fact} (u)$. 
	The first algorithm enumerates all pairs $(x,y) \in A^* \times A^*$ and checks whether $\varphi(xwy) = \varphi(z)$. This can be done for each such pair in view of Proposition \ref{wp}. The algorithm terminates if and only if $v$ is a factor of $u$. 

	\par 
	The second algorithm recursively enumerates all finite monoids $F$ which are quotients of $M$ and for each such $F$ all homomorphisms $\psi$ from $M$ to $F$. This can be done for each $F$ by considering all (finitely many) maps from $A$ to $F$, extending them to $A^*$ and checking if all the relations $R$ belong to the kernel of this map. For each $\psi: M \rightarrow F$ the algorithm checks whether $\psi(v) \in \text{Fact}(\psi(u))$, which can be done in the finite monoid $F$, and terminates if this is not true. We claim that this algorithm terminates if and only $v \notin \text{Fact}(u)$. Indeed, if the algorithm terminates, then $\psi(v) \notin \text{Fact}(\psi(u))$ for some $\psi$ and $F$, and so $v \notin \text{Fact}(u)$. Suppose now that $v \notin \text{Fact}(u)$. By Corollary \ref{fr}, the singleton $\{u\}$ is recognizable, and so the syntactic congruence $\sim_u$ has finite index, i.e., $M / \sim_u$ is finite. Take $F=M / \sim_u$ and $\psi = \; \sim_u$. Then $\psi(v) \notin \text{Fact}( \psi(u))$. Indeed, otherwise we would have $u \sim_u xvy$ for some $x,y \in M$, so $xvy=u$, and $v \in \text{Fact}(u)$, which is not true. Thus we have constructed the desired algorithm.
	
\end{proof}

\begin{proposition}\label{factors}
	Let $M$ be a finitely presented finite-$\J$-above monoid, and $u \in M$. Then the set of factors of $u$ is computable. 
\end{proposition}

\begin{proof}
	Let  {\rm Mon}$\langle A \mid R\rangle$ be the finite presentation for $M$ and $\varphi: A^* \rightarrow M$ be the corresponding epimorphism. Start computing all the words representing factors of $u$ in $A, A^2,A^3$ etc. Note that we can decide if a word represents a factor of $u$ by Proposition \ref{f}, and decide if two representatives correspond to the same element of $M$ by Proposition \ref{wp}. Since $u$ has finitely many factors, for some $n$ it will happen that all the words of length $n$ representing factors of $u$ have shorter representatives. We claim that we have now computed representatives of all factors of $u$. Indeed, suppose that $w  \in A^{n+k}$ represents a factor $z$ of $u$, for some $k>0$, and there are no shorter representatives for $z$ than $w$. Let $w=w_1w_2$, where $w_1$ has length $n$. Then $w_1$ represents a factor of $u$, and so it admits a shorter representative $w_1'$, therefore, $z$ admits a shorter representative than $w$, namely $w_1'w_2$, a contradiction. Thus all minimal representatives of factors of $u$ have length less than $n$, and are computable.
\end{proof}

\begin{theorem}\label{thm1}
	Let $M$ be a finitely presented finite-$\J$-above monoid.  Let also $\varphi: A^* \rightarrow M$ be a homomorphism, and $L\subseteq A^*$ be a language such that for any rational language $R$ in $A^*$ (given by a finite automaton) it is decidable whether $L \cap R$ is empty or not. Then the membership problem for $\varphi(L)$ is decidable in $M$.

\end{theorem}

\begin{proof}
	We have $u \in \varphi (L)$ if and only if $\varphi^{-1}(u) \cap L$ is non-empty, and $\varphi^{-1}(u)$ is rational, since $\{ u \}$ is recognizable by Corollary \ref{fr}. Thus, it suffices to show that a finite automaton for $\varphi^{-1}(u)$ is effectively constructible.
	\par 
	
	 Indeed, consider the following deterministic finite automaton $C(u)$ over $A$: its vertex set is the set of all factors of $u$ (which is finite since $M$ is finite-$\J$-above), the only initial state is $1$, the only terminal state is $u$, and for every vertex $v$ of $C(u)$ and $a \in A$ there is an $a$-labelled edge starting from $v$ if and only if the element $v \varphi(a)$ is a factor of $u$, and in this case this edge ends in $v \varphi(a)$. 
Then $C(u)$ recognizes $\varphi^{-1}(u)$. Indeed, if $p$ is a successful path in $C(u)$, with label $a_{i_1}...a_{i_k}$, then by definition we have $\varphi(a_{i_1}) \ldots \varphi(a_{i_k})=u$, so $a_{i_1}...a_{i_k} \in \varphi^{-1}(u)$. And if $x=a_{j_1}...a_{j_l} \in \varphi^{-1}(u)$, then $\varphi(a_{j_1})\varphi(a_{j_2})...\varphi(a_{j_l})=u$, so    $\varphi(a_{j_1}), \varphi(a_{j_1})\varphi(a_{j_2}), \ldots, \varphi(a_{j_1})\varphi(a_{j_2})\ldots\varphi(a_{j_{l-1}})$ are all factors of $u$, and we have a successful path with label $x$ by definition.

\par 
Note that the automaton $C(u)$ is algorithmically constructible, since all the factors of $u$ can be computed by Proposition \ref{factors}, and (in the above notation) one can decide whether $v \varphi(a)$ is a factor of $u$, and if yes, to which factor it is equal exactly, since $M$ has solvable word problem by Proposition \ref{wp}. This proves the theorem.
\end{proof}

The next corollary features a class of languages  known as
context-free languages. They are usually defined through structures called context-free grammars or pushdown
automata. For the basic theory, the reader is referred to \cite{Ber}, \cite{HU}.

\begin{corollary}
 In the notations of Theorem \ref{thm1}, membership problem is solvable for $\varphi(L)$ in $M$ for any context-free language $L$.	
 In particular, rational subsets and so finitely generated submonoids have solvable membership problem in $M$.
\end{corollary}

\begin{proof}
	If $L$ is context-free and $R$ is rational, then $L \cap R$ is context-free and effectively constructible from $L$ and $R$, and it is decidable whether a context-free language is empty or not. Rational subsets are images of rational languages, which are context-free, and finitely generated submonoids are particular cases of rational subsets.
\end{proof}

\subsection{Inverse monoids with the finite-$\J$-above condition}

In the following results we consider finitely presented inverse monoids (i.e., inverse monoids which have a finite inverse monoid presentation, not a finite monoid presentation).

\begin{proposition}\label{prop3}
	Let $M$ be a finite-$\J$-above finitely presented inverse monoid. Then $M$ has solvable word problem.
\end{proposition}

\begin{proof}
	Since $M$ is finite-$\J$-above, the $\J$-classes of $M$ are all finite. This implies that the $\cal{R}$-classes of $M$ are also finite, and so the Sch\"utzenberger automaton $\A(u)$ is finite for every $u$ in $M$. By Proposition \ref{cons}, this means that $\A(u)$ is constructible for every $u$, and then Proposition \ref{schutz} (namely, the equivalence of the first and the last conditions) implies that the word problem is solvable in $M$.
\end{proof}

\begin{proposition}\label{prop4}
	Let $M$ be a finite-$\J$-above finitely presented inverse monoid. Then for every element $w$ of $M$ the set of factors of $w$ is computable. 
\end{proposition}

\begin{proof}

 We could adapt the proof of Propositions \ref{f} and \ref{factors}, but the proof below gives an algorithm of lower complexity.
 
 	We use the notation of Subsection \ref{im}. Denote also by $Fact(w)$ the set of all factors of $w \in M$.
	Let $u \in \wt{A}^*$. As we saw in the proof of Proposition \ref{prop3}, the Sch\"utzenberger automaton $\A(u)$ is finite, since $M$ is finite-$\J$-above. Let $q_0=(uu^{-1})\tau$ and $t=u\tau$ denote the initial and terminal vertices of $\A(u)$ respectively. Let also $\A'(u)$ be a finite automaton obtained from the Sch\"utzenberger automaton $\A(u)$ by making all the vertices initial and all the vertices terminal. We first claim that 
\begin{equation}\label{l}	
	L(\A'(u))= \{ v \in \wt{A}^*: \: v\tau \in \text{Fact} (u\tau) \}.
\end{equation}

 Suppose first that $v \in L(\A'(u))$, so $v$ is the label of a path in $\A(u)$ starting in some vertex $p$ and ending in $q$. Since $\A(u)$ is trim, there is a path from $q_0$ to $p$ in $\A(u)$, labelled by some $x$, and a path from $q$ to $t$ in $A(u)$, labelled by some $y$. We have $xvy \in L(\A(u))$, so $uu^{-1}xvy \tau= u\tau$, and thus $v\tau \in \text{Fact}(u \tau)$.
 \par 
On the other hand, if $v \tau \in \text{Fact} (u\tau)$, then there exist $x,y$ such that $xvy \tau = u \tau$, and so $uu^{-1}xvy \tau=uu^{-1}u\tau =u\tau$, and therefore there is a path from $q_0$ to $t$ in $\A(u)$ with label $xvy$. This means that $v$ is a label of a subpath of this path, which is a successful path in $\A'(u)$, so $v \in L(\A'(u))$. This shows that (\ref{l}) holds.

Let $\wt{A}^n$ denote the set of all words in $\wt{A}^*$ of length $n$. For a vertex $q$ in $\A'(u)$ let $B_{n,q} \subseteq L(\A'(u)) \cap \wt{A}^n$ consist of all words $v$ of length $n$ in $L(\A'(u))$ such that there is a path in $\A'(u)$ with label $v$ ending in $q$. Then we have, for every $q$ and every $n>0$, using (\ref{l}):
$$\bigcup_{i=0}^n B_{i,q} \tau \subseteq \bigcup_{i=0}^{n+1} B_{i,q} \tau \subseteq \text{Fact}(u \tau).$$
Since $\text{Fact}(u \tau)$ is finite, the sequence $\bigcup_{i=0}^n B_{i,q}$ stabilizes for each $q$, 
so there exists a minimal $m>0$ such that for every $q$ we have
$$\bigcup_{i=0}^m B_{i,q} \tau = \bigcup_{i=0}^{m+1} B_{i,q} \tau.$$
Note that each $B_{n,q}$ is finite and computable, and since the word problem is solvable in $M$ by Proposition \ref{prop3}, we can compute the minimal $m$ as above.
\par 
We now show that for every $q$ and $n$
\begin{equation}\label{bnq}
	B_{n,q}\tau \subseteq \bigcup_{i=0}^m B_{i,q} \tau.
\end{equation}
Indeed, proceed by induction on $n$, suppose that for $n \leq n_0$ the claim is proved for every $q$ and prove it for $n=n_0+1$. Let $w \in B_{n_0+1,q}$, so $w$ is a label of some path in $\A'(u)$ ending in $q$. Write $w=za$, where $a \in \wt{A}$. Since the automaton $\A'(u)$ has all the vertices as initial and terminal, we have $z \in L(\A'(u))$, and so $z \in B_{n_0,r}$, where $r$ is the corresponding vertex of $\A(u)$ on the above path, connected to $q$ with an edge labelled by $a$. By the induction hypothesis, we have $B_{n_0,r}\tau \subseteq \bigcup_{i=0}^m B_{i,r} \tau$, so there exists $x \in  \bigcup_{i=0}^m B_{i,r}$ such that $z \tau x$. Since $x$ is the label of a path in $\A'(u)$ ending in $r$, $xa$ is the label of a path in $\A'(u)$ ending in $q$, so $xa \in \bigcup_{i=0}^{m+1} B_{i,q}$, so 
$$w \tau = za \tau =xa \tau \in  \bigcup_{i=0}^{m+1} B_{i,q} \tau = \bigcup_{i=0}^{m} B_{i,q} \tau,$$
as desired, and so (\ref{bnq}) holds.
\par
Thus the sequences $\bigcup_{i=0}^n B_{i,q} \tau$ stabilize from $n=m$ for every $q$. 
Let $S=\bigcup_q \bigcup_{i=0}^m B_{i,q}$, then $S$ is finite and computable.
Note that every element of $L(\A'(u))$ belongs to $B_{n,q}$ for some $n$ and $q$. Thus
$$ \text{Fact}(u \tau)=L(\A'(u))\tau= \bigcup_{q} \bigcup_{i=0}^{\infty} B_{i,q}\tau=\bigcup_{q} \bigcup_{i=0}^{m} B_{i,q}\tau=S\tau,$$ hence $\text{Fact}(u\tau)$ is computable.

\end{proof}

\begin{theorem}\label{thm2}
	Let $M$ be a finite-$\J$-above finitely presented inverse monoid, $\varphi: A^* \rightarrow M$ be a homomorphism, and $L\subseteq A^*$ be a language such that for any rational language $R$ in $A^*$ it is decidable whether $L \cap R$ is empty or not. Then the membership problem for $\varphi(L)$ is decidable in $M$.
\end{theorem}

\begin{proof}
	It follows from Propositions \ref{prop3} and \ref{prop4} as in the proof of Theorem \ref{thm1}.
\end{proof}

\begin{corollary}\label{cor2}
 In the notations of Theorem \ref{thm2}, membership problem is solvable for $\varphi(L)$ in $M$ for any context-free language $L$.	
 In particular, rational subsets and so finitely generated submonoids have solvable membership problem in finite-$\J$-above finitely presented inverse monoids.
\end{corollary}

Note that free inverse monoids are not finitely presented as monoids, so the above results do not follow directly from Theorem \ref{thm1}.

\section{Fragile words and the eraser morphism}

\subsection{Fragile words: definition and motivation}
Let $n \geq 2$ and let $FG_n=FG_A=FG(a_1,a_2,\ldots,a_n)$ denote the free group of rank $n$, with a free basis $A=\{ a_1, \ldots, a_n \}$. For an element $y$ of $FG(a_1, a_2, \ldots, a_n)$ and $1 \leq i \leq n$ denote by $y_{a_i}$ the element obtained from $y$ by deleting all occurrences of $a_i$ and its inverse. Note that $y_{a_i}$ can be considered as an element of $FG(a_1, a_2, \ldots, a_n)$ or as an element of $FG(a_1, \ldots, a_{i-1}, a_{i+1}, \ldots, a_n)$. 
Note that $(y_{a_i})_{a_j}=(y_{a_j})_{a_i}$, 
$(y_{a_i})_{a_i}=y_{a_i}$, $(y^{-1})_{a_i}=(y_{a_i})^{-1}$ and  $(uv)_{a_i}=u_{a_i}v_{a_i}$, for all $i,j=1,\ldots,n$.
\begin{definition}[Fragile words]
Following \cite{fragile}, we call an element $1\neq w \in FG(a_1,\ldots,a_n)$ {\it fragile} if $w_{a_i}=1$ for all $i=1,\ldots,n$. Note that in \cite{fragile} such elements were called strongly fragile. Denote the set of all fragile words in $FG_n$ by $Fr_n$.
\end{definition}

One of the main motivations to study fragile words comes from the connection with ``picture-hanging puzzles" and Brunnian links, see the introduction to our paper and \cite{puzzles}, \cite{GG} for more details. 

There is another motivation to study fragile words from \cite{fragile} that comes from the theory of automaton groups, i.e., groups defined by actions of certain alphabetical transducers on regular rooted trees. Let us start by recalling the notion of transducer (see, for instance, \cite{Ber}). 
 A (finite) \emph{transducer} is a quadruple $\mathcal{A} = (Q, A, \mu, \lambda)$, where:
\begin{itemize}
\item $Q$ is a finite set, called the set of \emph{states};
\item $A$ is a finite set, called the \emph{alphabet};
\item $\mu: Q \times A \to Q$ is the transition map or \emph{restriction};
\item $\lambda: Q\times A \to A$ is the output map or \emph{action}.
\end{itemize}


A very convenient way to represent a transducer is through its Moore diagram: this is a directed labelled graph whose vertices are
identified with the states of  $\mathcal{A}$ 
For every state $q \in Q$ and every letter $a \in A$, the diagram
has an arrow from $q$ to $\mu(q,a)$ labelled by $a|\lambda(q, a)$. The transducer
$\mathcal{A}$ contains a \emph{sink} $e\in Q$, if one has $\lambda(e, a)=a$ and $\mu(e,a)=e$ for any $a\in A$. The transducer  $\mathcal{A}$  is said to be {\it invertible} if, for all $q \in Q$, the transformation $\lambda(q, \cdot) : A \to A$ is a permutation of $A$.  If $\mathcal{A}$ is invertible, one can define the maps $\lambda$ and $\mu$ on $Q^{-1}$, the set of formal inverses of $Q$, by exchanging the input and the output in the automaton, i.e., for every state $q^{-1} \in Q^{-1}$ and every letter $a \in A$, the diagram has an arrow from $q^{-1}$ to $(\mu(q,a))^{-1}$ labelled by $\lambda(q, a)|a$. Finally, the maps $\lambda$ and $\mu$ can be naturally extended to $Q^*\times A^*$ by using the following recursive rules:
\begin{eqnarray}\label{eq: action}
\lambda(q, aw) = \lambda(q,a)\lambda(\mu(q, a), w), \ \ \lambda(qq', a) = \lambda(q,\lambda(q', a))
\end{eqnarray}
\begin{eqnarray}
\mu(q, aw) = \mu(\mu(q, a), w), \ \ \mu(qq', a) = \mu(q,\lambda(q', a))\mu(q', a)
\end{eqnarray}
for all $q \in Q$, $q' \in Q^*$, and $a \in A$, $w \in A^*$. We assume that when a transducer has a sink, it is supposed to be reachable from any state, i.e., for every $q\in Q$ there is $u\in A^*$ such that $\mu(q,u)=e$.

\begin{definition}
Given an invertible transducer $\cal{A}$, the automaton group $G=G(\mathcal{A})$
is the permutation group of $A^*$ generated by the states
$\{q: q\in Q\}$ with the operation defined by (\ref{eq: action}), where two elements $g,h\in \tilde{Q}^{*}$
represent the same element in  $G$ if
$$
\lambda(g,u)=\lambda(h,u)\mbox{ for every }u\in A^*.
$$
\end{definition}

The following proposition shows the crucial role that fragile words play in automata group theory (see \cite{fragile} for more details). The idea is to extend the alphabet $A$ to $A\cup Q$ and add a sink in case the transducer does not have one (we write $Q\cup\{e\}$). The maps $\lambda$ and $\mu$ keep the original action on $A$ and they are defined on $Q\cup\{e\}$ as follows
\begin{eqnarray*}
\lambda(q,p)=p \ \ \forall q,p\in Q\cup\{e\},
\end{eqnarray*}
\begin{eqnarray}\label{formuladouble}
\forall q\in Q \ \ \mu(q,q)=e \ \ \textrm{ and }\ \ \mu(q,p)=q \ \forall p\in (Q\cup\{e\})\setminus\{q\}.
\end{eqnarray}
We denote this new invertible transducer by $\mathcal{A}'$. There is a connection with fragile words and minimal defining relations of the group $G(\mathcal{A}')$. Roughly speaking if $w$ is a minimal defining relation of $G(\mathcal{A}')$, then if $Q'$ is the set of states appearing in the word $w$, then $w_q=1$ in $FG_{Q'}$ for all $q\in Q'$, i.e., there is a subset $Q'$ of $Q$ such that $w$ is fragile in $FG_{Q'}$. More precisely we have the following proposition.

\begin{proposition}\label{prop: fragile and relations}
With the above notation, if $G(\mathcal{A}')$ is not free, then for any shortest non-trivial relation $w=1$ of  $G(\mathcal{A}')$ there exists $Q'\subseteq Q$ with $w\in\wt{Q'}^{*}$ such that $w$ is a fragile word in $FG_{Q'}$.
\end{proposition}

\begin{proof}
First of all we observe that, if $w$ is a relation of an automata group $G(\mathcal{A})$ then $\mu(w,u)$ is also a relation. In fact $\lambda(w,v)=v$ for every $v\in A^{\ast}$ and so, in particular
$$
\lambda(w, uv)=u\lambda(\mu(w,u),v)=uv.
$$
This implies $\lambda(\mu(w,u),v)=v$ for every $v\in A^{\ast}$.
Now in the group $G(\mathcal{A}')$ take any shortest non-trivial relation $w \in \wt{Q}^{*}$, and let $Q'$ be the support of the word $w$, i.e., the set of states (with their inverses) that appear in $w$. Then, by \eqref{formuladouble} and the fact that the sink $e$ acts like the identity we deduce that if from the word $\mu(w,q)$ we erase all the occurrences of $e$ and $e^{-1}$, then we obtain another defining relation $w'$ with shorter length than $w$ and so $\oo{w'}=1$, and this occurs for any $q\in Q'$, i.e., $w$ is fragile in $FG_{Q'}$.
\end{proof}

\subsection{Eraser morphism for free groups}

\par 

Recall that $y_{a_i}$ is the word obtained from $y$ by deleting all the occurrences of $a_i$ and $a_i^{-1}$. We can think of the maps $y \rightarrow y_{a_i}$ as homomorphisms $$\theta_i: FG(a_1,\ldots, a_n) \rightarrow FG(a_1, \ldots, a_{i-1}, a_{i+1}, \ldots, a_n)$$ given by $(w)\theta_i=w_{a_i}$, for all $i=1,\ldots,n$. Note that $ w \langle \langle a_i \rangle \rangle = w_{a_i} \langle \langle a_i \rangle \rangle$ holds for every $w \in FG(a_1,\ldots,a_n)$. It follows that the kernel of $\theta_i$ is the normal closure of $a_i$, and the set of all fragile words in $F_n$ is the intersection of the kernels of all $\theta_i$, $i=1,\ldots,n$. We conclude that the set of all fragile words in $F_n$ forms a subgroup which is the intersection of normal closures of the generators: 
$$Fr_n=\langle \langle a_1 \rangle \rangle \cap \langle \langle a_2 \rangle \rangle \cap \ldots \cap \langle \langle a_n \rangle \rangle.$$

In this context the {\it eraser morphism} is the homomorphism $$\epsilon: FG(a_1, a_2, \ldots, a_n) \rightarrow FG(a_2,\ldots,a_n)\times FG(a_1,a_3,\ldots,a_n)\times \ldots \times FG(a_1,\ldots,a_{n-1})=G_n$$ given by 
$(x)\epsilon=((x)\theta_1,\ldots,(x)\theta_n)$. Then the set of all fragile words $Fr_n$ is the kernel of the eraser morphism $\epsilon$.

It is not difficult to see that $Fr_n$ is not finitely generated: it follows from the well-known fact that in free groups non-trivial normal finitely generated subgroups are of finite index (see, for instance, \cite{B}), and the observation that the subgroup $Fr_n$ has infinite index, since each $\langle \langle a_i \rangle \rangle$ has infinite index, and is non-trivial.
\par 
 The Schreier graph of $\langle \langle a_i \rangle \rangle$, denoted by $S_n^i$, can be obtained from the Cayley graph of  $F(a_1, \ldots, a_{i-1}, a_{i+1}, \ldots, a_n)$ (with respect to the generating set $\{a_1, a_1, \ldots, a_{i-1}, a_{i+1}, \ldots, a_n \}$) by adding a loop at every vertex labelled by $a_i$, and the Schreier graph for $Fr_n$ is the product graph $S_n$ of the Schreier graphs $S_n^i$, $1 \leq i \leq n$, see \cite[Section 9]{KM}.

\par 

Note that $S_2$ is isomorphic to the Cayley graph of $\mathbb{Z}^2$, i.e. is a grid with edges labelled by $a_1$ and $a_2$, and so $Fr_2$ is just the commutator subgroup of $FG_2$. 
However, for bigger $n$ the situation is much more complicated. In particular, it's not difficult to see that for $n \geq 3$  $Fr_n$ is strictly contained in the commutator subgroup of $FG_n$. For example, the word $a_1a_2a_3a_1^{-1}a_2^{-1}a_3^{-1}$ is in the commutator subgroup of $FG_3$, but not fragile. 

\par 


Let $K_n \leq G_n$ be the image of the eraser morphism $\epsilon$, so 
 $$K_n \cong FG_n / Fr_n.$$
We have a distinguished generating set $X_n$ of $K_n$ consisting of the elements $(a_1)\epsilon=(1,a_1,\ldots,a_1),(a_2)\epsilon=(a_2,1,a_2,\ldots,a_2), \ldots, (a_n)\epsilon=(a_n,\ldots,a_n,1)$. It follows that the Cayley graph of $K_n$ with respect to the generating set $X_n$ is isomorphic to the Schreier graph $S_n$. 

It is interesting to study the subgroups $K_n$, i.e. the images of eraser homomorphisms, in order to understand better the structure of fragile words in $FG_n$ for $n \geq 3$. Note that in general finitely generated subgroups in direct products of (two or more) free groups can behave quite wildly, in particular, they can even have unsolvable membership problem \cite{Mih}. However, we  show that $K_n$ has membership problem solvable by a simple algorithm. Moreover, $K_n$ is undistorted (i.e., quasi-isometrically embedded) in $G_n$. We also show that $K_n$ is not finitely presented for $n \geq 3$.


Recall that the word metric on a finitely generated group $G$ with respect to a finite generating set $S$ is defined as follows: the distance between $g,h \in G$ is the minimal length of a word expressing the element $g^{-1}h$ as a product of  elements in $S \cup S^{-1}$.
A subgroup $H$ in a group $G$ is called {\it undistorted} if there exists $C>0$ and some finite generating sets $X$ of $G$ and $Y$ of $H$ such that for every $h \in H$ we have $|h|_Y \leq C|h|_X$, where $|h|_X$, $|h|_Y$ denote the length of $h$ in the word metric given by $X$ and $Y$ respectively. 
A subgroup $H$ is undistorted in $G$ if and only if its natural embedding into $G$ is a quasi-isometric embedding  for the word metrics $|\cdot |_Y$ and $|\cdot |_X$. In particular, being undistorted does not depend on the choice of the generating sets $X$ and $Y$. Note that undistorted subgroups always have solvable membership problem. See \cite[page 506]{BH} for more information about subgroup distortion.

\begin{theorem}\label{T1}
	Suppose that $n \geq 2$. Let $g=(w^{(1)},w^{(2)},\ldots, w^{(n)})$ be an element of $G_n$, in the above notations. Then $g \in K_n$ if and only if  $w^{(i)}_{a_j}=w^{(j)}_ {a_i}$ for all $i,j=1, \ldots, n$, $i \neq j$. 
	\\ In particular, the membership problem for $K_n$ in $G_n$ is solvable for all $n \geq 2$. Moreover, $K_n$ is undistorted in $G_n$ for all $n \geq 2$.
\end{theorem}

\begin{proof}



By definition of $K_n$, an element $g=(w^{(1)},w^{(2)},\ldots, w^{(n)})$ belongs to $K_n$ if and only if there exists $w \in FG_n$ such that $w^{(i)}=w_{a_i}$ for all $i=1, \ldots, n$. If such $w$ exists, it follows that
 $w^{(i)}_{a_j}=w^{(j)}_ {a_i}$ for all $i,j=1, \ldots, n$, $i \neq j$.

Suppose now that $w^{(i)}_{a_j}=w^{(j)}_{a_i}$ for all $i,j=1, \ldots, n$. Note that $w^{(i)}$ contains no $a_i$, so $w^{(i)}_{a_i}=w^{(i)}$, for all $i=1, \ldots, n$.
We start constructing the desired word $w$ in $FG_n$ by induction, defining elements $x^{(1)},\ldots,x^{(n)}$. First let $x^{(1)}=w^{(1)}$. Then $x^{(1)}_{a_1}=w^{(1)}_{a_1}=w^{(1)}$. Now let 
 $$x^{(2)}=x^{(1)}(x_{a_2}^{(1)})^{-1}w^{(2)}=w^{(1)}(w_{a_2}^{(1)})^{-1}w^{(2)}.$$ Then, using our condition that $w^{(1)}_{a_2}=w^{(2)}_{a_1}$, we have 
 $$x^{(2)}_{a_1}=w^{(1)}_{a_1}((w_{a_2}^{(1)})^{-1})_{a_1}w^{(2)}_{a_1}=w^{(1)}_{a_1}((w_{a_1}^{(2)})^{-1})_{a_1}w^{(2)}_{a_1}=w^{(1)}_{a_1}(w_{a_1}^{(2)})^{-1}w^{(2)}_{a_1}=w^{(1)}_{a_1}=w^{(1)},$$
 
 $$x^{(2)}_{a_2}=  w^{(1)}_{a_2}(w_{a_2}^{(1)})^{-1}w^{(2)}_{a_2}=w^{(2)}_{a_2}=w^{(2)}.$$

Suppose by induction that we have constructed $x^{(k)}$, $k<n$, such that $x^{(k)}_{a_i}=w^{(i)}$ for $1 \leq i \leq k$.
We now define $x^{(k+1)}$ as follows:

$$x^{(k+1)}=x^{(k)}(x^{(k)}_{a_{k+1}})^{-1}w^{(k+1)}.$$

Now, for all $i=1, \ldots, k$ we have, using induction hypothesis and our condition that $w^{(i)}_{a_j}=w^{(j)}_{a_i}$: 

\begin{align*}
x^{(k+1)}_{a_i}=x^{(k)}_{a_i}((x^{(k)}_{a_{k+1}})^{-1})_{a_i}w^{(k+1)}_{a_i}=x^{(k)}_{a_i}((x^{(k)}_{a_i})^{-1})_{a_{k+1}}w^{(k+1)}_{a_i}=x^{(k)}_{a_i}(w^{(i)}_{a_{k+1}})^{-1}w^{(k+1)}_{a_i}= \\
x^{(k)}_{a_i}(w^{(k+1)}_{a_i})^{-1}w^{(k+1)}_{a_i}=x^{(k)}_{a_i}=w^{(i)}, 
\end{align*}

$$x^{(k+1)}_{a_{k+1}}= x^{(k)}_{a_{k+1}}(x^{(k)}_{a_{k+1}})^{-1}w^{(k+1)}_{a_{k+1}}= w^{(k+1)}_{a_{k+1}}=w^{(k+1)}.$$ 

This shows that for $x^{(n)}$ we have $x^{(n)}_{a_i}=w^{(i)}$ for all $1 \leq i \leq n$, and so we can take $w=x^{(n)}$, as desired.

Finally, it is easy to see that in the above construction the length of $w$ is linear in terms of the sum of the lengths of $w^{(1)},w^{(2)},\ldots, w^{(n)}$. This exactly means that $K_n$ is undistorted in $G_n$.

\end{proof}

For example,
consider $(w^{(1)},w^{(2)},w^{(3)})=(bcbc^2,ac^2ac,a^2b^2)$. Then $x^{(1)}=bcbc^2$, $x^{(2)}=(bcbc^2)c^{-3}(ac^2ac)=bcbc^{-1}ac^2ac$, and $w=x^{(3)}=(bcbc^{-1}ac^2ac)(a^{-2}b^{-2})(a^2b^2)$, which has the desired projections.

Note also that, even though the subgroups $K_n$ are quasi-isometrically embedded into $G_n$, the constants corresponding to these embeddings might grow exponentially fast in the above construction with the increase of $n$. One can ask if one can construct $w$ from $w^{(1)},w^{(2)},\ldots, w^{(n)}$ in a more optimal way. Indeed, it would be interesting to know what's the minimal possible length of the preimage in $FG_n$ of a given element in $K_n$.

\begin{theorem}\label{T2}
	The group $K_n$ is not finitely presented for $n \geq 3$.
\end{theorem}

\begin{proof}
	We deduce this result from a general theorem about finitely presented subgroups in direct products of free groups in \cite{BHMS}, namely Theorem D. We use some notions from \cite{BHMS}. First note that $K_n$ is a {\it subdirect} product, i.e., the projection of $K_n$ to each factor of $G_n$ is surjective. Indeed, for every element $x \in FG(a_2,a_3, \ldots, a_n)$ we can consider an element $x_0$ in $FG(a_1,a_2,\ldots,a_n)$ which is the same as a word as $x$, and then $\varphi(x_0)$ gives the desired element in $K_n$ with projection $x$ to the first component, and similar for other components. \par 
	Furthermore, $K_n$ is a {\it full} subdirect product, meaning that the intersection of $K_n$ with each of the factors is non-trivial. Indeed, it follows from Theorem \ref{T1} that the intersection of $K_n$ with the first factor of $G_n$ is precisely the subgroup of all fragile words in $FG(a_2,a_3, \ldots,a_n)$, which is non-empty, and similarly for other components.
	\par 
	Note also that, since $K_n$ is a subgroup of a direct product of free groups, it is residually free, i.e., every non-trivial element can be mapped via a homomorphism to a non-trivial element of some free group. Since $K_n$ is a full subdirect product, its natural embedding into $G_n$ is neat in the sense of \cite{BHMS} (as defined just before Theorem D), so if $K_n$ was finitely presented, the images of the projections of $K_n$ to pairs of factors in $G_n$ would all have finite index in those direct products of two free groups (we apply the implication (1) implies (5) in Theorem D of \cite{BHMS}, and use that free groups are certainly among limit groups, i.e. are fully residually free). However, we now show this is not the case.
	\par Consider the projection $L_n$ of $K_n$ to the direct product of the first two factors of $G_n$. Let $(x,y)$ be an element of $L_n$, where $x \in FG(a_2,\ldots,a_n)$, $y \in FG(a_1,a_3,\ldots,a_n)$. Then by Theorem \ref{T1} we have $x_{a_2}=y_{a_1}$. So if $M_n$ denotes the subgroup of  $FG(a_2,\ldots,a_n) \times FG(a_1,a_3,\ldots,a_n)$ consisting of all elements $(x,y)$ with $x_{a_2}=y_{a_1}$, then $L_n \subseteq M_n$. However, $n \geq 3$ and the elements $(1,a_3^k)$ are clearly in different cosets of $M_n$ in $G_n$ for different $k$, so $M_n$ has infinite index in $G_n$, and so also $L_n$ has infinite index in $G_n$, a contradiction. Thus $K_n$ is not finitely presented for $n \geq 3$.
\end{proof}

    We remark that these results almost certainly allow straightforward generalizations to some other groups apart from free groups.

\subsection{Eraser morphism for inverse monoids}

Let $\P = {\rm Inv}\langle A \mid R\rangle$ be a finite presentation of an inverse monoid $M=\wt{A}^*/(\rho_A \cup R)^{\sharp}$, and fix an enumeration $a_1,\ldots,a_m$ of the elements of $A$. For $i = 1,\ldots,m$, write $A_i = A \setminus \{ a_i \}$ and let $\theta_i:\wt{A}^* \to \wt{A_i}^*$ be the homomorphism defined by
$$a\theta_i = \left\{
\begin{array}{ll}
a&\mbox{ if }a \in \wt{A_i}\\
1&\mbox{ if }a \in \{ a_i,a_i\inv \}
\end{array}
\right.$$
We define also $\oo{\theta}_i: \wt{A}^* \times \wt{A}^* \to \wt{A_i}^* \times \wt{A_i}^*$ by $(u,v)\oo{\theta}_i = (u\theta_i,v\theta_i)$.

Let $\epsilon_i^{\P}: \wt{A}^*/(\rho_A \cup R)^{\sharp} \to \wt{A_i}^*/(\rho_{A_i} \cup R\oo{\theta}_i)^{\sharp}$ be defined by
$$(w(\rho_A \cup R)^{\sharp})\epsilon_i^{\P} = (w\theta_i)(\rho_{A_i} \cup R\oo{\theta}_i)^{\sharp}.$$
We claim that $\epsilon_i^{\P}$ is a well-defined surjective homomorphism.

Indeed, we may define a homomorphism $\p_{i}: \wt{A}^* \to \wt{A_i}^*/(\rho_{A_i} \cup R\oo{\theta}_i)^{\sharp}$ by $w\p_i = (w\theta_i)(\rho_{A_i} \cup R\oo{\theta}_i)^{\sharp}$. Since $\ker\p_i$ is a congruence, it suffices to show that 
\beq
\label{wde}
\rho_{A} \cup R \subseteq \ker\p_i.
\eeq

Let $u,v \in \wt{A}^*$. Since $(uu\inv u)\theta_i = (u\theta_i)(u\theta_i)\inv (u\theta_i)\, \rho_{A_i} \, u\theta_i$, we get
$$(uu\inv u)\p_i = ((uu\inv u)\theta_i)(\rho_{A_i} \cup R\oo{\theta}_i)^{\sharp} = (u\theta_i)(\rho_{A_i} \cup R\oo{\theta}_i)^{\sharp} = u\p_i.$$
Similarly, $(uu\inv vv\inv)\p_i = (vv\inv uu\inv)\p_i$ and so $\rho_{A} \subseteq \ker\p_i$.

On the other hand, given $(r,s) \in R$, we have $(u\theta_i,v\theta_i) \in R\oo{\theta}_i$ and therefore $u\p_i = v\p_i$. Thus (\ref{wde}) holds and so $\epsilon_i^{\P}$ is a well-defined surjective homomorphism. Henceforth we put $M_{i}=\wt{A_i}^*/(\rho_{A_i} \cup R\oo{\theta}_i)^{\sharp}$. 

In the context we may define the {\em eraser morphism} as the morphism:
$$\epsilon^{\P}:M \to \prod_{i=1}^m M_{i}$$
defined by 
$$(w(\rho_A \cup R)^{\sharp})\epsilon^{\P} = ((w(\rho_A \cup R)^{\sharp})\epsilon_i^{\P})_i.$$
Note that the enumeration of the letters of $A$ is irrelevant since we get isomorphic direct products in the image.

Our main objective is to study the eraser homomorphism, namely its image and its kernel.

\subsubsection{The membership problem for the image of the eraser morphism}




Since a direct product of finitely many finite-$\J$-above monoids is finite-$\J$-above, and free inverse monoids are finite-$\J$-above, Corollary \ref{cor2} implies the following.
\begin{corollary}
\label{decimage}
Let $\P = {\rm Inv}\langle A \mid R\rangle$ be a finite presentation with $A = \{ a_1,\ldots,a_m\}$.  If $\wt{A_i}^*/(\rho_{A_i} \cup R\oo{\theta}_i)^{\sharp}$ is finite-$\J$-above for $i = 1,\ldots,m$, then the membership problem is decidable for ${\rm Im}\,\epsilon^{\P} \leq \prod_{i=1}^m M_i$. In particular, this holds if $\P$ is a free inverse monoid. 
\end{corollary}

Note that, unlike the situation for free groups, in this case decidability of the membership problem for the image of the eraser morphism is a particular case of a much more general decidability result, as in Corollary \ref{cor2}.

The following example shows that $\wt{A}^*/(\rho_{A} \cup R)^{\sharp}$ finite-$\J$-above does not imply $\wt{A_i}^*/(\rho_{A_i} \cup R\oo{\theta}_i)^{\sharp}$ finite-$\J$-above.

\begin{example}
\label{fja}
The presentation ${\rm Inv}\langle a,b \mid aa\inv = b\rangle$ defines a finite-$\J$-above inverse monoid, but the presentation ${\rm Inv}\langle a \mid aa\inv = 1\rangle$ does not.
\end{example}

Indeed, the first presentation is clearly equivalent to ${\rm Inv}\langle a \mid \emptyset \rangle$, hence it defines the monogenic free inverse monoid, which is finite-$\J$-above. However, the second presentation defines the bicyclic monoid, where each element is a factor of any other element. Since the bicyclic monoid is infinite, it is not finite-$\J$-above.

\medskip

The next example shows that the condition characterizing ${\rm Im}\,\epsilon^{\P}$ in the free group case is not sufficient for free inverse monoids.

\begin{example}
\label{fgn}
Let $A = \{ a_1,a_2,a_3 \}$, $x_1 = a_2a_3$, $x_2 = a_3a_1$ and $x_3 = a_1a_2$. 
Then $x_i\theta_j = x_j\theta_i$ for all $i,j \in \{ 1,2,3\}$ but $(x_1\rho_{A_1}, x_2\rho_{A_2},x_3\rho_{A_3}) \notin {\rm Im}\,\epsilon^{\P}$.
\end{example}

Indeed, suppose that $(x_1\rho_{A_1}, x_2\rho_{A_2},x_3\rho_{A_3}) = (w\rho_A)\epsilon^{\P}$ for some $w \in \wt{A}^*$. Clearly, $w \neq 1$, so let $b$ denote the first letter of $w$.
\bi
\item
if $b\in \{a_1,a_1\inv \}$, we contradict $(w\theta_2)\rho_{A_2} = x_2\rho_{A_2}$;
\item
if $b\in \{a_2,a_2\inv \}$, we contradict $(w\theta_3)\rho_{A_3} = x_3\rho_{A_3}$;
\item
if $b\in \{a_3,a_3\inv \}$, we contradict $(w\theta_1)\rho_{A_1} = x_1\rho_{A_1}$.
\ei
Therefore $(x_1\rho_{A_1}, x_2\rho_{A_2},x_3\rho_{A_3}) \notin {\rm Im}\,\epsilon^{\P}$ as claimed.

In what follows we consider an inverse monoid $M$ having finite $\mathcal{R}$-classes and we put $\tau=(\rho_{A} \cup R)^{\sharp}$ and  $\tau_i=(\rho_{A_i} \cup R\oo{\theta}_i)^{\sharp}$ for $i=1,\ldots, n$. By Proposition~\ref{cons} each Sch\"utzenberger automaton $\A(u)$ is finite and constructible. Therefore, from a combinatorial point of view, it is natural to explore what is the relationship between the Sch\"utzenberger automata of $M$ and the ones in $M_{i}$. To describe this we we need to define the following natural operation on automata. 

\begin{definition}[$a_{i}$-contracted automaton]For an inverse automaton $\A$, the $a_{i}$-contracted automaton $\chi_{a_i}(\A)$ in the inverse automaton obtained in the following way. Identify all the states $p,q$ such that $p\mapright{a_{i}} q$ is an edge in $\A$, and erase all the loops labeled by $a_{i}$ (and consequently also $a_{i}^{-1}$). This operation generates a new involutive automaton $\B$ on the alphabet $\wt{A_{i}}$ that is not in general deterministic. Then, the $a_{i}$-contracted automaton is $\chi_{a_i}(\A)=\fold(\B)$. Let $\pi: \A\to \chi_{a_i}(\A)$ and $\pi': \A\to \B$ be the natural homomorphisms (in the category of inverse automata).\end{definition}

A word $g \in \wt{A}^*$ is said to be a {\em Dyck word} on $A$ if $g$ is reducible to the identity in $FG_{A}$. We denote by $D_A$ the language of all Dyck words on $A$.
 We have the following lifting lemma.
\begin{lemma}\label{lem: lifting}
With the above notation. Let $\pi(p) \mapright{u} \pi(q)$ be a path in $\chi_{a_i}(\A)$ and write $u = u_1\ldots u_k$ with $u_1,\ldots,u_k \in \wt{A_{i}}$. Then there exists a path $p\mapright{\overline{u}}q$ in $\A$ and a factorization

$$\overline{u}=g_0u_{1}g_{1}\ldots u_{k}g_k$$

in $\wt{A}^*$ with $g_j\theta_i \in D_{A_i}$ for $j = 0,\ldots,k$. In particular, $\oo{u}\p_i\le u\tau_i$ and so the inclusion 

$$L[\chi_{a_i}(\A)]\subseteq \{u\in \wt{A_{i}}^{*}: \oo{u}\p_{i}\le u  \tau_i\mbox{ for some }\oo{u}\in L[\A]\}$$

holds.

\end{lemma}

\begin{proof}

If an involutive automaton $\C'$ is obtained from an involutive automaton $\C$ by folding two edges $\bullet \mapleft{a} \bullet \mapright{a} \bullet$, then every path in $\C'$ can be lifted to a path in $\C$ by inserting factors $a\inv a$. Since $\chi_{a_i}(\A)$ is obtained by starting from $\B$ and successively folding edges, the lifting still holds but we must insert Dyck words (these are the words obtained by successively inserting factors of the form $a\inv a$ $(a \in \wt{A})$). Thus we can lift the path $\pi(p) \mapright{u} \pi(q)$ to some path $\pi'(p) \mapright{u'} \pi'(q)$ in $\B$, where 

$$u' =g'_0u_{1}g'_{1}\ldots u_{k}g'_k$$

and $g'_j \in D_{A_i}$ for $j = 0,\ldots,k$. Now we lift this path to a path $p\mapright{\overline{u}}q$ in $\A$ by inserting factors of the form $a_i^r$ $(r \in \mathbb{Z})$. Such insertions change the Dyck words $g'_j$ into the required $g_j$.

Let us prove the last claim of the lemma. Let $\alpha, \beta$ the initial and final state of $\A$, respectively. Let $u\in L(\chi_{a_i}(\A))$, by the first statement of the lemma we may lift a path $\pi(\alpha)\mapright{u}\pi(\beta)$ in $\chi_{a_i}(\A)$ to a path $\alpha\mapright{\oo{u}} \beta$ in $\A$ where

$$\overline{u}=g_0u_{1}g_{1}\ldots u_{k}g_k$$

and $g_j\theta_i \in D_{A_i}$ for $j = 0,\ldots,k$.

Now, by applying the eraser morphism $\epsilon_i^{\P}$ to $\oo{u}$ we get $\oo{u}\p_{i}=(\oo{u}\tau)\epsilon_i^{\P}=(\oo{u}\theta_{i})\tau_{i}$. Since each $g_{j}\theta_{i} \in D_{A_i}$, we get $(g_{j}\theta_{i})\tau_{i}\in E(M_{i})$, from which we deduce 
$$\oo{u}\p_{i}=(\oo{u}\tau)\epsilon_i^{\P}=(\oo{u}\theta_{i})\tau_{i}\le u\tau_{i}$$
and this completes the proof of the lemma. 
\end{proof}
We have the following lemma.
\begin{lemma} \label{lem: closure steph}
Let $M=\wt{A}^{*}/\tau$ be an inverse monoid with presentation ${\rm Inv}\langle A \mid R\rangle$ and let $\A$ be a connected inverse automaton then 
$$
L(\cl_{R}(\A))=\{w\in \wt{A}^{*}: w\tau\ge u\tau,\mbox{ for some }u\in L(\A) \}
$$
Moreover, if $Cl_{R}(\A)\simeq \A(u)$ for some Sch\"utzenberger automaton $\A(u)$, then there is a word $v\in L(\A)$ such that $v\tau=u\tau$.
\end{lemma}
\begin{proof}
The first statement is \cite[Lemma 3.4]{Steph98}. Since by Proposition~\ref{prop: isomorphic} we have $L[\A(u)]=L[Cl_{R}(\A)]$ we may deduce the following equality
$$\{ w\tau \mid w\tau \geq u\tau\} = \{ w\tau \mid w\tau \geq v\tau \mbox{ for some }v \in L(\A)\},$$ hence $u\tau \geq v\tau$ for some $v \in L(\A)$, i.e., $v\tau=u\tau$. 
\end{proof}
We have the following proposition.
\begin{proposition}\label{prop: sch as closure}
For any $w\in \wt{A}^{*}$, let $\A(w\theta_i)$ and $\A(w)$ be the Sch\"utzenberger automaton of the word $w\theta_i$, $w$ with respect to the presentation ${\rm Inv}\langle A_{i} \mid R\oo{\theta}_i\rangle$, ${\rm Inv}\langle A \mid R\rangle$, respectively. Then, 
$$
\A(w\theta_i)\simeq \cl_{R\oo{\theta}_i}(\chi_{a_i}(\A(w)))
$$
\end{proposition}
\begin{proof}
By Proposition~\ref{prop: isomorphic} we have $L(\A(w))=\{x\in \wt{A}^*: x\tau\ge w\tau\}$. We claim that the following fact:
$$
\mbox{if }x\tau\ge w\tau,\mbox{ then }(x\theta_i)\tau_i\ge (w\theta_i)\tau_i
$$
holds. This follows from the fact that $\epsilon_i^{\P}: \wt{A}^*/(\rho_A \cup R)^{\sharp} \to \wt{A_i}^*/(\rho_{A_i} \cup R\oo{\theta}_i)^{\sharp}$ is a well defined homomorphism, hence if $x\tau\geq w\tau$ we clearly get $\epsilon_i^{\P}(x\tau)\geq \epsilon_i^{\P}(w\tau)$, i.e., $(x\theta_i)\tau_i\ge (w\theta_i)\tau_i$. From Lemma~\ref{lem: lifting} and the previous fact we deduce the following inclusions:
\begin{eqnarray*}
L(\chi_{a_i}(\A(w))) &\subseteq &\{u\in \wt{A_{i}}^{*}: u\tau_{i}\ge (x\theta_i)\tau_i, \mbox{ for some }x\in L[\A(w)] \}\subseteq \\
&\subseteq&  \{u\in \wt{A_{i}}^{*}: u\tau_{i}\ge (w\theta_i)\tau_i\}
\end{eqnarray*}
Now, since $w\theta_i\in L(\chi_{a_i}(\A(w)))$, by Lemma~\ref{lem: closure steph} we immediately deduce that
$$
L(\cl_{R\oo{\theta}_i}(\chi_{a_i}(\A(w)))
)=\{u\in \wt{A_{i}}^{*}: u\tau_{i}\ge (w\theta_{i})\tau_i \}=L(\A(w\theta_i))
$$
i.e., $\A(w\theta_i)\simeq\cl_{R\oo{\theta}_i}(\chi_{a_i}(\A(w)))$ by Proposition~\ref{prop: isomorphic}.
\end{proof}
Our aim is to characterize those $n$-tuples $(u_1\tau_1, u_2\tau_2, \ldots, u_n\tau_n)\in \prod_{i=1}^m M_{i}$ that are in the image ${\rm Im}\,\epsilon^{\P} $. We provide a combinatorial characterization depending on the Sch\"utzenberger automata $\A(u_i)$ with respect to the presentations {\rm Inv}$\la A_i \mid R\oo{\theta_i} \ra$, for $i=1,\ldots, n$. We first consider the inverse automaton $\wh{\A}(u_i)=(Q_{i}, \alpha_i, \beta_i, E_i)$ obtained from $\A(u_i)$ by adding all the loops $q\mapright{a_{i}}q$, $q\mapright{a_{i}^-1}q$ for each state $q$ of $\A(u_i)$. Now consider the product automaton:
$$
\mathcal{P}(u_1,\ldots, u_n)=\wh{\A}(u_1)\times \wh{\A}(u_2)\times \ldots \times \wh{\A}(u_n)
$$
i.e., the inverse automaton obtained considering the set of states $Q_1\times \ldots \times Q_n$ and transitions $(q_1, \ldots, q_n)\mapright{a} (p_1, \ldots, p_n)$ whenever $q_i\mapright{a} p_i$ is a transition in $\wh{\A}(u_i)$ for all $i=1,\ldots, n$, and the initial and final state $(\alpha_1,\ldots, \alpha_n)$, $(\beta_1,\ldots, \beta_n)$, respectively. It is not difficult to see that the language $L(\mathcal{P}(u_1,\ldots, u_n))$ is formed by words $w\in \wt{A}^{*}$ such that $w\theta_i\in L({\A}(u_i))$ for all $i=1,\ldots, n$. Thus, it is formed by all the words $w$ such that $(w\theta_{i})\tau_i\ge u_{i}\tau_{i}$ for all $i=1,\ldots, n$. Conversely, if $w\in \wt{A}^{*}$ is a word such that $(w\theta_i)\tau_i \ge u_i\tau_i$ for all $i=1,\ldots, n$, then $w\in L(\mathcal{P}(u_1,\ldots, u_n))$. Thus the following equality:
$$
L(\mathcal{P}(u_1,\ldots, u_n))=\{w\in \wt{A}^{*}: (w\theta_i)\tau_i \ge u_i\tau_i\mbox{ for all }i=1,\ldots, n\}
$$ 
holds. We have the following combinatorial characterization of the elements in ${\rm Im}\,\epsilon^{\P}$.
\begin{proposition}\label{prop: combinatorial characterization}
With the above notation, an $n$-tuple $(u_{1}\tau_{1}, \ldots, u_{n}\tau_{n})\in {\rm Im}\,\epsilon^{\P} $ if and only if 
$$
\cl_{R\oo{\theta_i}}\left(\chi_{a_i}(\mathcal{P}(u_1,\ldots, u_n))\right)=\A(u_i)\mbox{ for all }i=1,\ldots, n.
$$
\end{proposition}
\begin{proof}
Suppose that $(u_{1}\tau_{1}, \ldots, u_{n}\tau_{n})\in {\rm Im}\,\epsilon^{\P}$ and let $w\in\wt{A}^*$ such that $(w\tau) \epsilon^{\P}=(u_{1}\tau_{1}, \ldots, u_{n}\tau_{n})$. By Lemma \ref{lem: lifting} we have 
\begin{align*}
L(\chi_{a_i}(\mathcal{P}(u_1,\ldots, u_n)))&\subseteq \{u\in \wt{A_{i}}^{*}: s\p_{i}\le u  \tau_i\mbox{ for some }s\in L(\A(u_i)\}\subseteq \\
&\subseteq \{s\in \wt{A_i}^{*}: s\tau_{i}\ge u_{i}\tau_{i}\}
\end{align*}
Since $u_{i}\tau_{i}=w \epsilon_i^{\P}=(w\theta_i)\tau_{i}$ we have $w\in L(\mathcal{P}(u_1,\ldots, u_n))$. In particular $w\theta_i\in L(\chi_{a_i}(\mathcal{P}(u_1,\ldots, u_n)))$ with $(w\theta_i)\tau_{i}=u_{i}\tau_{i}$. Therefore, by Lemma~\ref{lem: closure steph} we get:
$$
L\left(\cl_{R\oo{\theta_i}}\left(\chi_{a_i}(\mathcal{P}(u_1,\ldots, u_n))\right)\right)= \{s\in \wt{A_i}^{*}: s\tau_{i}\ge u_{i}\tau_{i}\}=L(\A(u_i))
$$
i.e., $\cl_{R\oo{\theta_i}}\left(\chi_{a_i}(\mathcal{P}(u_1,\ldots, u_n))\right)=\A(u_i)$.
\\
On the other hand, suppose that $\cl_{R\oo{\theta_i}}\left(\chi_{a_i}(\mathcal{P}(u_1,\ldots, u_n))\right)=\A(u_i)$. Let $\alpha_i, \beta_i$ be the initial and final states of $\chi_{a_i}(\mathcal{P}(u_1,\ldots, u_n))$, respectively. By Lemma~\ref{lem: closure steph} there is a word $z_{i}\in L(\chi_{a_i}(\mathcal{P}(u_1,\ldots, u_n)))$ such that $z_{i}\tau_{i}=u_{i}\tau_{i}$. Now by Lemma \ref{lem: lifting} we may lift the path $\alpha_i\mapright{z_{i}} \beta_i$ of $\chi_{a_i}(\mathcal{P}(u_1,\ldots, u_n))$ to a path $\alpha\mapright{\oo{z}_{i}} \beta$ in $\mathcal{P}(u_1,\ldots, u_n)$ (where $\alpha, \beta$ are the initial and final states of $\mathcal{P}(u_1,\ldots, u_n)$, respectively), for some $\oo{z}_{i}\in\wt{A}^*$ with $(\oo{z}_{i}\theta_i)\tau_{i}\le z_{i}\tau_{i}$. Consider the following word:
$$
w=(\oo{z}_{1}\oo{z}_{1}^{-1})(\oo{z}_{2}\oo{z}_{2}^{-1})\ldots (\oo{z}_{n}\oo{z}_{n}^{-1})\oo{z}_1
$$
note that $w\in L[\mathcal{P}(u_1,\ldots, u_n)]$. We claim that $(w\tau) \epsilon^{\P}=(u_{1}\tau_{1}, \ldots, u_{n}\tau_{n})$. We have 
$$
((ww^{-1})\theta_i)\tau_i\le ((\oo{z}_{i}\oo{z}_{i}^{-1})\theta_i)\tau_{i}\le (z_iz_{i}^{-1})\tau_{i}=(u_iu_i^{-1})\tau_i
$$
Now, since $\alpha\mapright{w}\beta$ is a path in $\mathcal{P}(u_1,\ldots, u_n)$ we have $w\in L[\wh{\A}(u_i)]$, from which we deduce $w\theta_i\in L[\A(u_i)]$. Thus, since the inequality $((ww^{-1})\theta_i)\tau_i\le (u_iu_i^{-1})\tau_i$ holds, by Proposition~\ref{prop: isomorphic} we get $((ww^{-1})\theta_i)\tau_i= (u_iu_i^{-1})\tau_i$, and so since the inequality $(w\theta_i)\tau_i\ge u_i\tau_i$ is equivalent to $u_i\tau_i=(u_iu_i^{-1})\tau_i (w\theta_i)\tau_i$ we get our claim $u_i\tau_i=(w\theta_i)\tau_i$. Hence, $(w\tau)\epsilon_i^{\P}=u_i\tau_i$, i.e., $(w\tau) \epsilon^{\P}=(u_{1}\tau_{1}, \ldots, u_{n}\tau_{n})$. 
\end{proof}
We recall that an idempotent $e$ covers $g$ if $g<e$ and if there is no other idempotent $f\neq e, g$ such that $g\le f\le e$. 
We say that an inverse monoid $M$
has the \emph{finite covering property} if any idempotent $g\in E(M)$ satisfies the following conditions: 
\begin{itemize}
\item
there exist finitely many idempotents $e_1,\ldots,e_n$ covering $g$;
\item
any idempotent $f > g$ satisfies $f \geq e_i$ for some $i \in \{ 1,\ldots,n\}$.
\end{itemize}
Note that the property is satisfied if the interval $[g,1] = \{ f \in E(M) \mid g \leq f \}$ is finite for every $g \in E(M)$.

We have the following main result.
\begin{theorem}
With the above notation, let $M$ be an inverse monoid such that each $i$-th component $M_i$ has finite $\mathcal{R}$-classes. Suppose that
each $M_i$ has the finite covering property and for each $g$ it is possible to effectively calculate the covering idempotents $e_1,\ldots, e_m$ of $g$.
Then, the membership problem for ${\rm Im}\,\epsilon^{\P}$ is decidable.
\end{theorem}
\begin{proof}
Let $(u_{1}\tau_{1}, \ldots, u_{n}\tau_{n})\in  \prod_{i=1}^m M_i$. Since each component $M_i$ has finite $\mathcal{R}$-classes, then by Proposition~\ref{cons} we may construct each Sch\"utzenberger automaton $\A(u_i)$ with respect to the presentation ${\rm Inv}\langle A_i \mid R\oo{\theta_i}\rangle$. Thus, we may effectively construct the product automaton 
$$
\mathcal{P}(u_1,\ldots, u_n)=\wh{\A}(u_1)\times \wh{\A}(u_2)\times \ldots \times \wh{\A}(u_n)
$$
where $\wh{\A}(u_i)$ is obtained from $\A(u_i)$ by adding all the loops $q\mapright{a_{i}}q$, $q\mapright{a_{i}^-1}q$ for each state $q$ of $\A(u_i)$. Now, by Proposition~\ref{prop: combinatorial characterization} it is enough to check if the following property
\begin{equation}\label{eq: cond to check}
\cl_{R\oo{\theta_i}}\left(\chi_{a_i}(\mathcal{P}(u_1,\ldots, u_n))\right)=\A(u_i)
\end{equation}
holds for each $i=1,\ldots, n$. Note that the automaton $\B_i=\chi_{a_i}(\mathcal{P}(u_1,\ldots, u_n))$ may also be effectively constructed. Let $\alpha$, $\alpha'$ be the initial states of $\A(u_i), \B_i$, respectively. Change the final state of $\A(u_i)$ to $\alpha$, in this way we are considering the Sch\"utzenberger automaton $\A(u_iu_i^{-1})$. Consider the automaton $\C_i$ obtained from $\B_i$ by taking as final state the initial state $\alpha'$. To check the condition expressed in equation (\ref{eq: cond to check}) it is enough to check whether both $\A(u_iu_i^{-1})=\cl_{R\oo{\theta_i}}(\C_i)$ and $u_i\in \cl_{R\oo{\theta_i}}\left(\chi_{a_i}(\mathcal{P}(u_1,\ldots, u_n))\right)$ hold. Indeed, if $\A(u_iu_i^{-1})$ is not isomorphic to $\cl_{R\oo{\theta_i}}(\C_i)$, then (\ref{eq: cond to check}) does not hold. Otherwise, if $\A(u_iu_i^{-1})=\cl_{R\oo{\theta_i}}(\C_i)$, since $\A(u_iu_i^{-1})$ is finite, then $\cl_{R\oo{\theta_i}}(\C_i)$ is finite, and in turn $\cl_{R\oo{\theta_i}}\left(\chi_{a_i}(\mathcal{P}(u_1,\ldots, u_n))\right)$ is also finite, and thus it may be effectively computed. Thus checking condition (\ref{eq: cond to check}) is now equivalent to check if $u_i\in \cl_{R\oo{\theta_i}}\left(\chi_{a_i}(\mathcal{P}(u_1,\ldots, u_n))\right)$, and this may be effectively checked. The following algorithms depending on the conditions stated in the theorem check whether the condition 
\begin{equation}\label{eq: 2 cond to check}
\A(u_iu_i^{-1})=\cl_{R\oo{\theta_i}}(\C_i)
\end{equation}
holds or not. 

Now take $(u_iu_i^{-1})\tau_i$ and compute all the idempotents $v_1\tau_i,\ldots, v_m\tau_i$ covering $(u_iu_i^{-1})\tau_i$. Then perform the following procedure:
\begin{itemize}
\item Check if $L(\C_i)\subseteq L(\A(u_iu_i^{-1}))$. If the answer is negative, then the algorithm exits by stating that equality (\ref{eq: cond to check}) does not hold. 
\item Otherwise, build all the Sch\"utzenberger automata $\A(v_j)=(Q_j, \alpha_j, \alpha_j, E_j)$ for all $j=1,\ldots, n$ and change the final states to get a new automaton $\mathcal{D}_j=(Q_j, \alpha_j, Q_j, E_j)$ with the property that $L(\mathcal{D}_j)=\{u: \alpha_j\mapright{u}p\mbox{ for some }p\in Q_j\}$. Consider also the automaton $\C_i'=(Q, q_0, Q, E)$ obtained from $\C_i=(Q,q_0,q_0,E)$ by changing the set of final states, in this way the automaton $\C_i'$ recognizes the language $\{u: q_0\mapright{u}p\mbox{ for some }p\in Q\}$. The algorithm now checks whether the following inclusion 
$$
L(\C'_i)\subseteq \bigcup_{j=1}^m L(\mathcal{D}_j)
$$
holds. If this last condition occurs, then the algorithm exits reporting that equality (\ref{eq: cond to check}) does not hold, otherwise we would have the contradiction $(u_iu_i^{-1})\tau_i\ge(v_jv_j^{-1})\tau_i $ for some $j=1,\ldots, m$. Otherwise, there is some $z\in L(\C'_i)\setminus (\bigcup_{j=1}^m L(\mathcal{D}_j))$. Therefore, since the word $zz^{-1}\in L(\C_i)\setminus (\bigcup_{j=1}^m L(\mathcal{D}_j))$, then we deduce that $(zz^{-1})\tau_i\ge (u_iu_i^{-1})\tau_i$ and $(zz^{-1})\tau_i<v_j\tau_i$ for all $j=1,\ldots, m$. Hence, $(zz^{-1})\tau_i= (u_iu_i^{-1})\tau_i$. Therefore, since $(zz^{-1})\tau_i= (u_iu_i^{-1})\tau_i$ and $L(\C_i)\subseteq L(\A(u_iu_i^{-1}))=\{v\in\wt{A}_i^*: v\tau_i\ge (u_iu_i^{-1})\tau_i\}$  by Lemma~\ref{lem: closure steph}  we may conclude that $\A(u_iu_i^{-1})=\cl_{R\oo{\theta_i}}(\C_i)$.
\end{itemize}
%
%
\end{proof}
Note that a free inverse monoid has finite $\mathcal{R}$-classes and satisfies the two conditions stated in the previous theorem. In what follows, we state some conditions on an inverse monoid $M$ in order for $M_i$ to have finite $\mathcal{R}$-classes. We first state the following general lemma. 
\begin{lemma}\label{lem: finiteness condition}
Let $\B$ be a finite inverse $\wt{A}$-automaton with initial state $\alpha$ and let $M=\wt{A}^{*}/\tau$ be an inverse monoid with presentation ${\rm Inv}\langle A \mid R\rangle$. Let $\alpha'$ be the corresponding initial state of $\cl_{R}(\B)$. If  $\B$ has the following property
$$
\{vv^{-1}\in\wt{A}^*: \alpha\mapright{vv^{-1}}\alpha \mbox{ is a path in }\B\}=\{vv^{-1}\in\wt{A}^*: \alpha'\mapright{vv^{-1}}\alpha' \mbox{ is a path in }\cl_{R}(\B)\}
$$
Then, $\cl_{R}(\B)$ is finite.
\end{lemma}
\begin{proof}
 For each state $q\in Q$ of $\cl_{R}(\B)$ we associate a subset $S_{q}$ formed by all the states $p$ of $\B$ such that $\alpha\mapright{u}p$ is a path in $\B$ whenever $\alpha\mapright{u} q$ is a path in $\cl_{R}(\B)$. Note that the property stated in the lemma ensures that $S_{q}\neq \emptyset $ for all the states $q$ in $\cl_{R}(\B)$. Furthermore, note that for all $q,q'$ with $q\neq q'$, $S_{q}\cap S_{q'}=\emptyset$. Indeed, assume contrary to our claim that there is a state $p\in S_{q}\cap S_{q'}$. Hence,  in $\cl_{R}(\B)$ there are paths $\alpha'\mapright{u} q$, $\alpha'\mapright{v} q'$, and in $\B$ we there are paths $\alpha\mapright{u} p$, $\alpha\mapright{v} p$. Hence, $\alpha\mapright{uv^{-1}}\alpha$ is a loop in $\B$, and since $\cl_{R}(\B)$ is the closure of $\B$, we have that $\alpha'\mapright{uv^{-1}}\alpha'$ is also a loop in $\cl_{R}(\B)$, i.e., $q=q'$, a contradiction. Therefore, the collection of subsets $S_{q}$, $q\in Q$ forms a partition of the set of states of $\B$ which is a finite set. Further, since $S_{q}\neq \emptyset$ for all $q\in Q$ we may conclude that $Q$ is finite, i.e., $\cl_{R}(\B)$ is finite. 
\end{proof}
Using the previous lemma it is possible to prove the following.
\begin{proposition}\label{prop: finite R-classes}
Assume that the inverse monoid $M$ has finite $\mathcal{R}$-classes. If for all $e\in E(M_i)$, the set of preimages $\{g\in M: \epsilon_i^{\P}(g)=e\}$ has a minimum idempotent, then $M_{i}$ has also finite $\mathcal{R}$-classes. 
\end{proposition}
\begin{proof}
Let $u\tau_i\in M_i$ we show that the Sch\"utzenberger automaton $\A(u)$ is finite by showing that $\A(uu^{-1})$ is finite. Consider the set $((uu^{-1})\tau_i) \p_i^{-1}$ of the preimages of the idempotent $(uu^{-1})\tau_i$. By the hypothesis in the statement, there is $x\in \wt{A}^*$ such that $(xx^{-1})\tau$ is minimum among the idempotents of the preimage $\{g: \epsilon_i^{\P}(g)=(uu^{-1})\tau_i\}$, i.e., $(zz^{-1})\tau \ge (xx^{-1})\tau$ for all $(zz^{-1})\in ((uu^{-1})\tau_i) \p_i^{-1}$. Now, consider the Sch\"utzenberger automaton $\A_A(xx^{-1})$ with respect to {\rm Inv}$\la A|R\ra$, and consider the corresponding $a_i$-contracted automaton $\B=\chi_{a_i}(\A_A(xx^{-1}))$. By Proposition~\ref{prop: sch as closure} we have $\cl_{R\oo{\theta_i}}(\B)$ is isomorphic to the Sch\"utzenberger automaton $\A((xx^{-1})\theta_i)$ with respect to ${\rm Inv}\langle A_{i} \mid R\oo{\theta}_i\rangle$. We claim that the following inclusion
\begin{equation}\label{eq: inclusion}
\{vv^{-1}: \alpha'\mapright{vv^{-1}}\alpha' \mbox{ is a path in }\cl_{R\oo{\theta_i}}(\B)\}\subseteq \{vv^{-1}: \alpha\mapright{vv^{-1}}\alpha \mbox{ is a path in }\B\}
\end{equation}
holds. Indeed, for any path $ \alpha'\mapright{vv^{-1}}\alpha'$ in $\cl_{R\oo{\theta_i}}(\B)=\A((xx^{-1})\theta_i)$ we may deduce that $(vv^{-1})\tau_i\ge ((xx^{-1})\theta_i)\tau_i$. Hence, we get the equality $(vv^{-1}(xx^{-1})\theta_i)\tau_i=((xx^{-1})\theta_i)\tau_i$. Now, since we have $(xx^{-1})\theta_i\in ((uu^{-1})\tau_i) \p_i^{-1}$, we also deduce that $(vv^{-1}(xx^{-1})\theta_i)\in ((uu^{-1})\tau_i) \p_i^{-1}$ holds. Thus, by the minimality of the idempotent $(xx^{-1})\tau$ we get $(vv^{-1}(xx^{-1})\theta_i)\tau \ge (xx^{-1})\tau$. Hence, there is a path $\alpha\mapright{vv^{-1}}\alpha$ in $\A_A(xx^{-1})$, and consequently there is also a path $\alpha\mapright{vv^{-1}}\alpha$ in $\B=\chi_{a_i}(\A_A(xx^{-1}))$, and this concludes the proof of the inclusion (\ref{eq: inclusion}). Now, since also the other inclusion holds, we get the equality
$$
\{vv^{-1}: \alpha'\mapright{vv^{-1}}\alpha' \mbox{ is a path in }\cl_{R\oo{\theta_i}}(\B)\}= \{vv^{-1}: \alpha\mapright{vv^{-1}}\alpha \mbox{ is a path in }\B\}
$$
and so by Lemma~\ref{lem: finiteness condition} we conclude that $\A((xx^{-1})\theta_i)$ is finite. Thus, since $((xx^{-1})\theta_i)\tau_i=(uu^{-1})\tau_i$ we may deduce that $\A(uu^{-1})$ is also finite. 
\end{proof}

\subsubsection{Kernel and connections with fragile words}

We consider now the congruence
$$\ker \epsilon^{\P} = \{ (x,y) \in \wt{A}^*/\tau \times \wt{A}^*/\tau\mid x\epsilon^{\P} = y\epsilon^{\P} \},$$
the {\em kernel} of $\epsilon^{\P}$. 

Now congruences on inverse monoids are usually described by the so-called kernel and trace. To avoid confusion with the kernel of an endomorphism, we shall use the notation $K(\tau)$ for the kernel of a congruence $\tau$ on an inverse monoid $M$, that is,
$$K(\tau) = \{ u \in M \mid u \,\tau\, u^2 \}.$$
On the other hand, the trace of $\tau$ is the restriction of $\tau$ to $E(M)$.

Let $\sigma_A$ denote the canonical free group congruence on $\wt{A}^*$ and let $\pi_A:FIM_A \to FG_A$ be the canonical epimorphism.

\begin{proposition}
\label{kefr}
Let $A = \{ a_1,\ldots,a_m\}$ and let $\epsilon^{\P}:FIM_A \to \prod_{i=1}^m FIM_{A_i}$ be the homomorphism induced by the canonical presentation ${\rm Inv}\langle A \mid \emptyset \rangle$. Let $w \in \wt{A}^*$.
Then $w\rho_A \in K({\rm ker}\,\epsilon^{\P})$ if and only if $w\sigma_A$ is a fragile word.
\end{proposition}

\begin{proof}
Let $\P' = {\rm Inv}\langle A \mid aa\inv = 1\; (a \in \wt{A}) \rangle$ be the canonical inverse monoid presentation defining $FG_A$. The canonical epimorphisms $\pi_{A_i}$ $(i = 1,\ldots,m)$, induce in the obvious way an epimorphism
$$\pi_{A_1} \times \ldots \times \pi_{A_m}:FIM_{A_1} \times \ldots \times FIM_{A_m} \to FG_{A_1} \times \ldots \times FG_{A_m}.$$
We claim that
\beq
\label{kefr1}
\epsilon^{\P}(\pi_{A_1} \times \ldots \times \pi_{A_m}) = \pi_{A}\epsilon^{\P'}.
\eeq

Indeed, let $w \in \wt{A}^*$. Then 
$$\begin{array}{lll}
(w\rho_A)\epsilon^{\P}(\pi_{A_1} \times \ldots \times \pi_{A_m})&=&((w\theta_1)\rho_{A_1}, \ldots, (w\theta_m)\rho_{A_m})(\pi_{A_1} \times \ldots \times \pi_{A_m})\\
&=&((w\theta_1)\sigma_{A_1}, \ldots, (w\theta_m)\sigma_{A_m}) = (w\sigma_A)\epsilon^{\P'}\\
&=&(w\rho_A)\pi_{A}\epsilon^{\P'}
\end{array}$$
and so (\ref{kefr1}) holds.

Assume that $w\rho_A \in K({\rm ker}\,\epsilon^{\P})$. Then $(w\rho_A)\epsilon^{\P} = (w^2\rho_A)\epsilon^{\P}$ and in view of (\ref{kefr1}) we get 
$$(w\sigma_A)\epsilon^{\P'} = (w\rho_A)\pi_{A}\epsilon^{\P'} = (w^2\rho_A)\pi_{A}\epsilon^{\P'} = (w^2\sigma_A)\epsilon^{\P'}.$$
Since $\im \epsilon^{\P'}$ is a group, we get $(w\sigma_A)\epsilon^{\P'} = 1$ and so $w\sigma_A$ is a fragile word.

Conversely, assume that $w\sigma_A$ is a fragile word, i.e. $(w\sigma_A)\epsilon^{\P'} = 1$. In view of (\ref{kefr1}), we get $((w\theta_i)\rho_{A_i})\pi_{A_i} = 1$ for $i = 1,\ldots,m$. But it is well known that $1\pi_{A_i}\inv = E(FIM_{A_i})$, hence $(w\theta_i)\rho_{A_i} \in E(FIM_{A_i})$ for $i = 1,\ldots,m$. Thus
$$(w\rho_A)\epsilon^{\P} = ((w\theta_1)\rho_{A_1}, \ldots, (w\theta_m)\rho_{A_m}) = ((w^2\theta_1)\rho_{A_1}, \ldots, (w^2\theta_m)\rho_{A_m}) = (w^2\rho_A)\epsilon^{\P}$$
and so $w\rho_A \in K({\rm ker}\,\epsilon^{\P})$ as required.
\end{proof}


\subsection{Acknowledgments}
We would like to thank Enric Ventura for helpful conversations in which he conjectured that Theorem \ref{T1} holds, and Montse Casals-Ruiz for pointing out the reference \cite{BHMS} which allowed us to prove Theorem \ref{T2}.

The first author thanks Austrian Science Fund project FWF P29355-N35.
The third author was partially supported by CMUP (UID/MAT/00144/2019), which is funded by FCT (Portugal) with national (MCTES) and European structural funds through the programs FEDER, under the partnership agreement PT2020.
The last author was partially supported by the ERC Grant 336983, by the Basque Government grant IT974-16, by the grants MTM2014-53810-C2-2-P and MTM2017-86802-P of the Ministerio de Economia y Competitividad of Spain, and by the grant 346300 for IMPAN from the Simons Foundation and the matching 2015-2019 Polish MNiSW fund.

\end{document}